\documentclass[11pt]{amsart}
\usepackage{amsmath}
\usepackage{amsfonts}
\usepackage{amssymb}
\usepackage{amsthm}
\usepackage{indentfirst}
\usepackage{graphicx}
\usepackage[left=2.7cm,right=2.7cm,top=3cm,bottom=3cm]{geometry}

\title[Stability of Gel'fand's inverse problem for Schr{\"o}dinger operators]{Stability of Gel'fand's inverse interior spectral problem for Schr{\"o}dinger operators}
\author{Jinpeng Lu}
\address{Jinpeng Lu: Department of Mathematics and Statistics, University of Helsinki, Finland} 
\email{jinpeng.lu@helsinki.fi}

\usepackage{hyperref}
\usepackage{xcolor}

\theoremstyle{theorem}
\newtheorem{lemma}{Lemma}[section]
\newtheorem{proposition}[lemma]{Proposition}
\newtheorem{coro}[lemma]{Corollary}
\newtheorem{theorem}[lemma]{Theorem}
\newtheorem{main1}{Theorem}

\newtheorem{main-coro}[main1]{Theorem}

\theoremstyle{definition}
\newtheorem{definition}[lemma]{Definition}

\newtheorem{remark}[lemma]{Remark}

\numberwithin{equation}{section}

\def\R{\mathbb R}

\def\N{\mathbb N}
\def\M{\mathcal{M}}
\def\h{\widehat}
\def\e{\varepsilon}
\def\CL{C_3}
\def\Cpot{C_9}
\def\Cd{C_7}

\def\Vol{{\rm vol}}
\def\rL{r_L}
\def\cLip{c_2}

\def \inj {{\rm inj}}
\def\ceigen {c_1}
\def \appnorm {\mathcal{L}^a}

\def \CHau {C_3}
\def\Cq {C_0}
\def \CLap {C_8}

\subjclass[2020]{35R30, 58J50, 53B21}

\begin{document}
\maketitle


\vspace{-6mm}

\begin{abstract}

We study Gel'fand's inverse interior spectral problem of determining a closed Riemannian manifold $(M,g)$ and a potential function $q$ from the knowledge of the eigenvalues $\lambda_j$ of the Schr{\"o}dinger operator $-\Delta_g + q$ and the restriction of the eigenfunctions $\phi_j|_U$ on a given open subset $U\subset M$, where $\Delta_g$ is the Laplace-Beltrami operator on $(M,g)$.
We prove that an approximation of finitely many spectral data on $U$ determines a finite metric space that is close to $(M,g)$ in the Gromov-Hausdorff topology, and further determines a discrete function that approximates the potential $q$ with uniform estimates.
This leads to a quantitative stability estimate for the inverse interior spectral problem for Schr{\"o}dinger operators in the general case.

\end{abstract}

\vspace{2mm}

\section{Introduction}

Gel'fand's inverse problem \cite{Ge} concerns the unique determination of a Schr{\"o}dinger operator on a bounded domain of $\mathbb{R}^n$ from its spectral data on the boundary.
The problem can be formulated as an inverse boundary spectral problem for the boundary spectral data, that is, the Dirichlet eigenvalues and the normal derivative of orthonormalized eigenfunctions on the boundary, see \cite{KKL,KKLM}.
It was proved in \cite{NSU}, and for equivalent formulations \cite{Novikov,RS}, that the inverse boundary spectral problem is uniquely solvable on given bounded domains of $\R^n$ based on the method of complex geometric optics developed in \cite{SyU}.

Gel'fand's inverse problem has a natural geometric formulation on Riemannian manifolds: can the Riemannian metric $g$ and the potential function $q$ be uniquely determined from the boundary spectral data of a Schr{\"o}dinger operator $-\Delta_g + q$ on a compact Riemannian manifold with boundary? Typically one does not assume to know the interior of the manifold itself, which needs to be determined as part of the problem.
In the 1990s, the combination of the Boundary Control method, introduced originally on $\R^n$ in \cite{B87}, and Tataru’s optimal unique continuation theorem established the unique solvability for the inverse problem for smooth Riemannian metrics \cite{BK,T95}.
The unique solvability for metrics of lower regularity (of Zygmund class $C^2_{*}$) was proved in \cite{AKKLT}.
Analogously, one could consider the same type of inverse problem on closed Riemannian manifolds (i.e., compact without boundary).
It was known by \cite{HLOS,KrKL,KLLY} that a closed Riemannian manifold $(M,g)$ is uniquely determined from the eigenvalues $\lambda_j$ of the Laplace-Beltrami operator $\Delta_g$ and the restriction of orthonormalized eigenfuctions $\phi_j|_U$ on a given (possibly small) nonempty open subset $U\subset M$.
The formulation is equivalent to an inverse problem for the wave equation from the source-to-solution map \cite{KKLM,Lassas}, and has applications in the fractional Calder{\'o}n problem on closed manifolds \cite{FGKU}.
Following \cite{BKL}, we refer to this formulation on closed Riemannian manifolds as the \emph{inverse interior spectral problem}.

\smallskip
We study the stability of Gel'fand's inverse problem on Riemannian manifolds.
The problem is ill-posed in the sense of Hadamard, as one can make large changes to the manifold while only cause small effects on the spectral data measured in a small subset. To stabilize the inverse problem, it is natural to impose \emph{a priori} bounds on geometric parameters and consider the stability issue in a compact class of manifolds.
Gel'fand's inverse spectral problem enjoys continuity in the general class of $n$-dimensional Riemannian manifolds with bounded diameter, sectional curvature and injectivity radius \cite{AKKLT,KLLY}, however with no information on the modulus of continuity.
With additional geometric assumptions, strong H{\"o}lder-type stability estimates for the determination of Riemannian metrics were obtained in \cite{SU98,SU05} when the metric is close to being simple (i.e., with convex boundary and no cut points).
In the case that the metric is Euclidean \cite{AS90,IS92,Sun} or close to a fixed simple metric \cite{BD10,BD11,Mon}, the determination of the potential function also exhibits H{\"o}lder-type stability.
For non-simple metrics, sensitivity results were established for non-trapping metrics in a given conformal class under a fold-regular condition for the conjugate points \cite{BZ}, and under the foliation condition \cite{SUV16} introduced in \cite{SUV}.
In a geometric class that relaxes the non-trapping assumption, a H{\"o}lder-type stability estimate for the determination of the potential was obtained in \cite{AG}, assuming the absence of conjugate points.

In the general case we do not expect a H{\"o}lder-type stability estimate to hold for Gel'fand's inverse problem. 
In fact, it is proved recently in \cite{BLOS} that a closely related inverse problem of recovering the potential $q$ in a wave equation with an obstacle is exponentially unstable, based on the methods established in \cite{KRS}.
The stability of reconstructing the Riemannian metric in Gel'fand's inverse spectral problem with zero potential $q=0$ is known to be at least double-logarithmic in general geometry \cite{BKL,BILL}.
The proofs were based on the quantitative unique continuation for the wave operator that was proved in \cite{BKL18,BKL16} and \cite{LL} independently.
Recently using quantitative unique continuation approach, a double-logarithmic stability estimate for recovering the potential $q$ from the source-to-solution map for the wave equation over $\R^{n}$ is obtained by \cite{FO1,FO2} for general configurations.
To the best of our knowledge, a stability estimate for Gel'fand's inverse problem with nonzero potential $q$ is still unknown in general geometry.
The main purpose of this paper is to prove a double-logarithmic stability estimate in the general case for the determination of both the Riemannian manifold and the potential $q$ in Gel'fand's inverse interior spectral problem for Schr{\"o}dinger operators.
Our proof is based on the quantitative unique continuation and a graph discretization of Laplacians, and provides an algorithm to reconstruct a discrete approximation of the Riemannian manifold and the potential.

\medskip
Let $(M,g)$ be a connected closed smooth Riemannian manifold of dimension $n\geq 2$ in a class with bounded geometry defined by
\begin{eqnarray} \label{def-class}
&{\rm diam}(M)\leq D, \quad |{\rm Sec}(M) |\leq K^2, \quad \Vol(M)\geq v_0>0, \\
& \|\nabla R(M) \| + \|\nabla^2 R(M) \| \leq K_2, \label{bound-C2}
\end{eqnarray}
where ${\rm diam}(M), \,{\rm Sec}(M), \,\Vol(M)$ denote the diameter, sectional curvature and Riemannian volume of $M$, and $\nabla^k R(M)$ is the $k$-th covariant derivative of the curvature tensor $R(M)$ of $M$.
We denote this class of Riemannian manifolds by $\M(n,D,K,K_2,v_0)$.

Let $q:M\to \R$ be a smooth function satisfying
\begin{equation}\label{bound-q}
\|q\|_{C^{0,1}(M)}\leq \Cq \, ,
\end{equation}
for some constant $C_0>0$.
We consider the Schr{\"o}dinger operator $-\Delta_g +q$, where $\Delta_g$ is the Laplace-Beltrami operator on $(M,g)$.
Denote by $\lambda_j$ the $j$-th eigenvalue of the Schr{\"o}dinger operator $-\Delta_g +q$, and by $\phi_j$ an eigenfunction corresponding to $\lambda_j$.
The eigenvalues are counted according to their multiplicities: $\lambda_1<\lambda_2 \leq \lambda_3 \leq \cdots $, with $\lambda_j \to +\infty$ as $j\to \infty$.
We assume that the eigenfunctions are orthonormalized so that $\{\phi_j\}_{j=1}^{\infty}$ is an orthonormal basis of $L^2(M)$.
Given any open subset $U\subset M$, the \emph{interior spectral data} of $-\Delta_g + q$ on $U$ refers to a collection of data
$$\Big( U, \big\{\lambda_j, \phi_j|_U \big\}_{j=1}^{\infty} \Big),$$
where $\phi_j|_U$ is the restriction of the eigenfunction $\phi_j$ on $U$.
Note that the choice of orthonormalized eigenfunctions in each eigenspace is not unique and different choices differ by orthogonal transformations.

\begin{definition} \label{def-data-close}
Suppose that the set $U\subset M$ and the Riemannian metric $g|_U$ on $U$ are given.
We say a collection of data $\{\lambda_j^a, \phi_j^a|_U\}$ is a $\delta$-approximation of the spectral data of $-\Delta_g +q$ on $U$, if there exists a choice of interior spectral data $\{\lambda_j, \phi_j|_U\}_{j=1}^{\infty}$ of $-\Delta_g +q$ on $U$ such that the following conditions are satisfied for all $j \leq \delta^{-1}\,$:
\begin{itemize}
\item[(1)] $\lambda_j \in \R$, $\phi_j^a|_U \in C^{0,1}(U)$;
\item[(2)] $|\lambda_j-\lambda_j^a| <\delta$;
\item[(3)] $\|\phi_j-\phi_j^a\|_{C^{0,1}(U)} < \delta$.
\end{itemize}
Let $M_1,M_2$ be two Riemannian manifolds, and $U_1,U_2$ be open subsets of $M_1,M_2$, respectively. Suppose that $U_1$ is isometric to $U_2$ via a Riemannian isometry $F$. 
For $i=1,2$, we say the spectral data of $M_i$ on $U_i$ are $\delta$-close, if the pull-back via $F$ of the spectral data of $M_2$ on $U_2$ is a $\delta$-approximation of the spectral data of $M_1$ on $U_1$.
\end{definition}

In particular, the finite spectral data $\{\lambda_j, \phi_j|_U\}_{j=1}^{J}$ is a $J^{-1}$-approximation of the complete spectral data $\{\lambda_j, \phi_j|_U\}_{j=1}^{\infty}$ on $U$ by Definition \ref{def-data-close}.

\smallskip
Our first main result is a reconstruction of both the Riemannian manifold $(M,g)$ and the potential function $q$ from an approximation of the interior spectral data of $-\Delta_g + q$ on a given nonempty open subset $U\subset M$.

\begin{main1}\label{main1}
There exist uniform constants $C_1,C_2,C_3>1$ 
such that the following holds.
Let $(M,g) \in \M(n,D,K,K_2,v_0)$ and $U\subset M$ be an open set containing a ball of radius $r_0>0$.
Suppose that the set $U$ and the Riemannian metric $g|_U$ on $U$ are given.
Let $q$ be a potential function on $M$ satisfying $\|q\|_{C^{0,1}(M)}\leq \Cq$.
Then for any $\e>0$, there exists $\delta>0$, such that any $\delta$-approximation $\{\lambda_j^a,\phi_j^a |_U\}$ of the spectral data of $-\Delta_g + q$ on $U$ determine a Riemannian manifold $(\widehat{M},\widehat{g})$ and numbers $\h{q}_i$, for $i=1,\cdots,I_0$, satisfying the following properties.

\begin{itemize}
\item[(1)] 
The Riemannian manifold $\widehat{M}$ is diffeomorphic to $M$ and satisfies
$$d_{L}\big( (M,g), (\widehat{M},\widehat{g}) \big)\leq C_1\e^{\frac{1}{12}},$$
where $d_{L}$ denotes the Lipschitz distance between metric spaces.

\item[(2)] 
There exists a $\CHau\e$-net $\{x_i\}_{i=1}^{I_0}$ in $M$ such that
$$\big| \, \h{q}_i - q(x_i) \big| \leq C_2\,\e^{\frac{1}{80n}}, \quad \textrm{ for all }i=1,\cdots,I_0.$$
\end{itemize}
The constants $C_1,C_2,C_3$ depend on $n,D,K,K_2,v_0,r_0,\Cq$, and
the choice of $\delta$ can be made explicit in $\e$, depending on the same set of parameters.
\end{main1}



We note that the estimates in Theorem \ref{main1} are uniform over the class $\M(n,D,K,K_2,v_0)$, independent of individual manifolds.
The definition of the Lipschitz distance is recalled in \eqref{def-Lip-distance}.
In particular, the estimate on the Lipschitz distance implies the same estimate on the Gromov-Hausdorff distance.
In fact, we will construct in Proposition \ref{prop-metric} a finite metric space which gives the existence of an $\e^{\frac18}$-isometry from the finite metric space to $(M,g)$.
This gives an approximation of $(M,g)$ in the sense of Gromov-Hausdorff.
Recall that a map $f:X\to Y$ between two metric spaces is called an $\e$-isometry if $f(X)$ is an $\e$-net in $Y$ and $| d_Y (f(x),f(x')) -d_X(x,x') | < \e$ for all $x,x'\in X$.
The improvement from the Gromov-Hausdorff distance to the Lipschitz distance is due to \cite{FIKLN}.
Our reconstruction of the potential function $q$ is based on a graph discretization of Laplacians, see Remark \ref{remark-Lap}.

\smallskip
A consequence of Theorem \ref{main1} is quantitative stability of the inverse interior spectral problem.

\begin{main-coro} \label{main-coro}
There exists a uniform constant $\delta_0>0$ such that the following holds.
For $i=1,2$, let $(M_i,g_i) \in \M(n,D,K,K_2,v_0)$ and $U_i\subset M_i$ be an open set containing a ball of radius $r_0>0$ in $M_i$. Suppose that $(U_1,g_1|_{U_1)}$ is isometric to $(U_2,g_2|_{U_2})$.
Let $q_i$ be a potential function on $M_i$ satisfying $\|q_i\|_{C^{0,1}(M_i)}\leq \Cq$.
If the spectral data of $-\Delta_{g_i}+q_i$ on $U_i$ are $\delta$-close for $\delta < \delta_0$, then $M_1$ is diffeomorphic to $M_2$, and 
$$d_{L}\big( (M_1,g_1) , (M_2,g_2) \big) \leq \omega(\delta),$$
where $\omega(\delta) = C_1 \big(\log|\log \delta| \big)^{-C_4}.$
Moreover, there exists a $\omega(\delta)$-isometry $\Psi:M_1\to M_2$ such that
$$\| \, q_1 - q_2 \circ \Psi  \|_{L^{\infty}(M_1)}\leq C_2 \, \omega(\delta) .$$
The constants $C_1,C_2,\delta_0$ depend on $n,D,K,K_2,v_0,r_0,\Cq$, and $C_4$ depends only on $n$.
\end{main-coro}

We note that the map $\Psi$ depends on $\delta$, and is constructed essentially by the correspondence between the $\CHau\e$-nets of $M_1, M_2$ stated in Theorem \ref{main1}.

\smallskip
Let us remark on the idea of proofs. 
For the reconstruction of the Riemannian metric, we mainly follow the quantitative version of the geometric Boundary Control method developed in \cite{KKL04,BKL,BILL} and applied to the first eigenfunction (that does not change sign) in Section \ref{sec-metric}.
With relaxed regularity conditions on the curvature tensor in our setting, we need to establish the bi-Lipschitz property of distance coordinates, Proposition \ref{coordinate-Lip}, using a different method adapted to lower regularity setting based on Toponogov's theorem inspired by \cite{Ivanov1,Ivanov2}.
This is discussed in Section \ref{sec-coordinate}. This consideration is crucial to us avoiding a third logarithm in the main result.
Then we can approximate the interior distance functions from an approximation of the spectral data by defining a slicing procedure, which further determines a distance $\h{d}$ on a finite $\e^{\frac12}$-net $X$ in $M$ such that $(X,\h{d})$ is close to $(M,g)$ in the Gromov-Hausdorff topology.
Note that for the purpose of reconstructing the potential later, it is necessary to modify the previously known procedures so that different slices of the manifold are disjoint.
The details are explained in Section \ref{subsec-slicing}.

For the reconstruction of the potential, we use a new approach based on graph discretizations.
Previous works in the literature require given smooth manifold structure to construct geometric optics solutions, which is not available for our problem.
In Section \ref{sec-potential}, we construct a discrete function on the finite space $X$ such that it approximates $q|_X$ with uniform estimates.
To this end, it suffices to construct an approximation of $\phi_1|_X$ and $(\Delta_g \phi_1) |_X$ for the first eigenfunction $\phi_1$.
The pointwise values of $\phi_1$ can be approximated essentially by averaging in a larger ball (of a radius lower order than the Gromov-Hausdorff approximation).
The Laplacian of $\phi_1$ can be approximated by a weighted graph Laplacian on $X$ of the form
$$(\Delta_X f)(x_i) =  \frac{C(n)}{\Vol(B(x_i,\rho)) \rho^2} \sum_{k:\, d(x_i,x_k)<\rho} \Vol(V_k) \big( f(x_k)-f(x_i) \big), \quad \textrm{ for }x_i \in X,$$
where $C(n)$ is a normalization constant and the sets $V_k$ form a partition of the manifold.
It is expected that this type of discrete operator can approximate the Laplace-Beltrami operator in the spectral sense \cite{BIK,BIK2,Lu}, and in probabilistic settings widely studied in machine learning, see e.g. \cite{AV25,BN,GGHS,GHL,V}.
However, the issue is that all quantities appearing in the discrete operator above are unknown since the manifold or metric is not given.
In Section \ref{sec-potential},
we prove that all these quantities can be approximated using the metric approximation $\h{d}$ and the modified slicing procedure in Section \ref{sec-metric} with uniform estimates.



\smallskip
It is not known to us if the dependence of our stability estimates on the covariant derivatives of the curvature tensor \eqref{bound-C2} can be removed.
The additional regularity conditions on the curvature tensor is needed for uniformly controlling the higher-order terms in the Taylor expansion of the first eigenfunction $\phi_1$ in geodesic normal coordinates, see Proposition \ref{appro-Lap}. We make the essential use of the property that the relative distortion of metric in geodesic normal coordinates is in the second order, which does not hold in harmonic coordinates.
Furthermore, it is desirable to develop quantitative stability methods for inverse problems under more general Ricci curvature bounds, due to potential applications in general relativity.
Gel'fand's inverse spectral problem is continuous in the class of Riemannian manifolds with bounded Ricci curvature, diameter, and injectivity radius bounded from below \cite{AKKLT}.
Inspired by the integral version of the Toponogov comparison theorem for Ricci curvature developed in \cite{Colding3}, we ask if a quantitative stability estimate for the inverse problem can be established in this class.

\smallskip
{\bf Question.} \emph{Does Theorem \ref{main-coro} hold in the class of $n$-dimensional closed Riemannian manifolds with bounded Ricci curvature, diameter, and injectivity radius bounded from below?}

\medskip
\noindent {\bf Acknowledgements.} I would like to thank Shouhei Honda and Anton Petrunin for helpful discussions on spectral convergence in RCD spaces and angle semi-continuity in Alexandrov~spaces.

\section{Preliminaries}

We start by making the following observation on the potential $q$ satisfying \eqref{bound-q}.
To prove the main results, it suffices to assume that 
\begin{equation} \label{bound-q-positive}
\Cq^{-1} \leq q \leq \Cq, \quad \|q\|_{C^{0,1}(M)}\leq \Cq, \quad \textrm{ for some constant }\Cq>1.
\end{equation}
This is because otherwise one can consider our problem for the potential $q+2\Cq$ (satisfying the conditions above with constant $3C_0$).
Recall the standard min-max formula that
$$\lambda_k =\inf_{Q_k} \sup_{v\in Q_k\setminus \{0\}} \frac{\int_M |\nabla v|^2 + \int_M q v^2}{\int_M v^2},$$
where the infimum is taken over all $k$-dimensional subspaces $Q_k$ of the Sobolev space $H^1(M)$.
In particular, under the condition \eqref{bound-q-positive}, we have
\begin{equation} \label{bound-lambda1}
\Cq^{-1} \leq \lambda_1 \leq \Cq,
\end{equation}
for the first eigenvalue $\lambda_1$. 
The first eigenvalue $\lambda_1$ is simple and has a positive first eigenfunction, see e.g. \cite[Thm. 8.38]{GT}.
From now on, we suppose that $\phi_1>0$ on $M$.
The following lemma states that there is a uniform lower bound for the $L^2$-normalized first eigenfunction $\phi_1$. This is due to \cite{Honda18}, see related spectral convergence results in \cite{Honda24, KK94}.

\begin{lemma} \label{lemma-eigen-positive}
Let $(M,g) \in \M(n,D,K,v_0)$ and $\|q\|_{C^{0,1}(M)}\leq \Cq$.
Let $\phi_1$ be the (positive) first eigenfunction of $-\Delta_g +q$ with $\|\phi_1\|_{L^2(M)}=1$.
Then there exists a uniform constant $\ceigen=\ceigen(n,D,K,v_0,\Cq)>0$ such that $\phi_1\geq \ceigen$.
\end{lemma}

\begin{proof}
We argue by contradiction. Suppose the claim is not true: then there exists a sequence of manifolds $M_i \in \M(n,D,K,v_0)$ and functions $q_i$ such that the following holds. There exists a sequence of points $y_i\in M_i$ such that $\phi_{1,i}(y_i)\to 0$, where $\phi_{1,i}$ is the first eigenfunction of $-\Delta_{M_i} + q_i$ on $M_i$ with $\|\phi_{1,i}\|_{L^2(M_i)}=1$.
By the compactness theorem \cite{A90}, there is a subsequence of $M_i$ converging to a $C^{1,\alpha}$-Riemannian manifold $M_0$ in the $C^{1,\alpha}$-topology.
Up to a subsequence, we may assume that $y_i$ converge to a point $y_0\in M_0$.
On each $M_i$, using the form of Laplacian in local $C^{1,\alpha}$-harmonic coordinates, $\|q_i\|_{C^{0,1}}\leq \Cq$ and the first eigenvalue $\lambda_1 \leq \|q\|_{C^0}$, uniform H\"older estimates apply to $\phi_{1,i}$, see \cite[Thm. 8.24]{GT}.
By the Arzela-Ascoli theorem, there exists a subsequence of $\phi_{1,i}$ converging to $\psi$ uniformly on $M_0$, and $q_i$ converging to a Lipschitz function $q_0$ uniformly on $M_0$.
In particular, since $y_i\to y_0$ we have $\psi(y_0)=0$.
Then by \cite[Thm. 1.5]{Honda18}, $\psi$ is the (positive) first eigenfunction of $-\Delta_{M_0}+q_0$. In particular $\psi>0$ on $M_0$, which is contradiction to $\psi(y_0)=0$.
\end{proof}

Our reconstruction method relies on the quantitative unique continuation.
Consider the wave equation on $(M,g)$ with a Lipschitz potential function $q$,
\begin{equation} \label{wave-Y-intro}
(\partial_t^2-\Delta_g+q) u= f,
\end{equation}
where $\Delta_g$ is the Laplace-Beltrami operator on $(M,g)$. 
The following is a quantitative unique continuation result for $C^2$-Riemannian manifolds \cite[Thm. 1.3]{LLY}.
Although this result was originally done for $q=0$, adding a zero-th order potential term does not affect the proof except that the constants would additionally depend on $\|q\|_{C^0}$, see \cite[Thm. 1.2]{BKL16}.

\begin{theorem}[\cite{LLY}]\label{uc-Y}
Let $(M,g)$ be a closed Riemannian manifold satisfying
$$
{\rm dim}(M)=n\geq 2, \quad  {\rm diam}(M)\leq D, \quad |{\rm Sec}(M)|\leq K^2, \quad \Vol(M)\geq v_0>0.
$$
Suppose that $u\in H^1(M\times[-T,T])$ is a solution of the wave equation \eqref{wave-Y-intro} with $f\in L^2(M\times [-T,T])$. Let $U\subset M$ be an open subset with smooth boundary.
If 
$$\|u\|_{H^1(M\times[-T,T])}\leq \Lambda,\quad \|u\|_{H^{1}(U\times [-T,T])}\leq \varepsilon_0,$$
then for $0<h<h_0$, we have
$$\|u\|_{L^2(\Omega(h))}\leq C_5 \exp(h^{-C_6})\frac{\Lambda}{\bigg(\log \big(1+\frac{h^{-1}\Lambda}{\|f\|_{L^2(M\times[-T,T])}+h^{-2}\varepsilon_0}\big)\bigg) ^{\frac{1}{2}}}\, .$$
The domain $\Omega(h)$ is defined by
$$\Omega(h) :=\big\{(x,t)\in (M \setminus U)\times [-T,T]: T-|t|-d(x,U) > \sqrt{h} \big\}.$$
The constants $C_5,h_0$ depend on $n,D,K,v_0,T,\|q\|_{C^0}$, and $C_6$ depends only on $n$. 
\end{theorem}

To specify the dependence on geometric parameters in the course of proofs, we clarify our notations for classes of manifolds.
Denote by $\M(n,D,K,v_0)$ the class of connected closed smooth Riemannian manifolds of dimension $n\geq 2$ satisfying
\begin{equation} \label{def-class-C0}
{\rm diam}(M)\leq D, \quad |{\rm Sec}(M) |\leq K^2, \quad \Vol(M)\geq v_0>0.
\end{equation}
In addition, we denote by $\M(n,D,K,K_1,v_0)$ the class of Riemannian manifolds in $\M(n,D,K,v_0)$ satisfying additionally 
\begin{equation}
\|\nabla R(M) \| \leq K_1.
\end{equation}
For $(M,g) \in \M(n,D,K,v_0)$, the injectivity radius is uniformly bounded from below due to \cite{Cheeger},
\begin{equation} \label{inj}
\inj(M) \geq i_0=i_0(n,D,K,v_0).
\end{equation}

The space $\M(n,D,K,v_0)$ of closed Riemannian manifolds is pre-compact in the $C^{1,\alpha}$-topology \cite{Gromov-book, Peters}.
Namely, given any sequence $(M_i,g_i) \in \M(n,D,K,v_0)$, there exists a subsequence of $(M_i,g_i)$ converging to a Riemannian manifold in the $C^{1,\alpha}$-topology.
A sequence of Riemannian manifolds $(M_i,g_i)$ is said to converge in the $C^{1,\alpha}$-topology to a Riemannian manifold $(M,g)$, if for sufficiently large $i$, there exist diffeomorphisms $F_i : M\to M_i$ such that $F_i^* g_i$ converges to $g$ in the $C^{1,\alpha}$-topology, see e.g. \cite[Chap. 10.3]{Petersen}.
The limit space of such a sequence is a $C^{1,\alpha}$-Riemannian manifold and has two-sided curvature bounds in the sense of Alexandrov, see e.g. \cite[Prop. 10.7.1]{BBI} and \cite[Prop. 11.3]{Geo4}.

For $x,y\in M$, we use the notation $|xy|$ for the Riemannian distance $d_M(x,y)$ or simply $d(x,y)$.
Denote by $B(x,r)$ the open geodesic ball of radius $r$ centered at $x$.
We denote by $[xy]$ any minimizing geodesic from $x$ to $y$.
For $x,y,z\in M$, we denote by $\angle yxz$ the angle between the tangent vectors of $[xy]$ and $[xz]$ in the tangent space $T_x M$.
Note that the angle may not be uniquely defined, depending on the choice of minimizing geodesics $[xy]$ and $[xz]$ if multiple minimizing geodesics exist.
This configuration of two minimizing geodesics starting at the same point is referred as a hinge.

Let $(X,d_X)$ and $(Y,d_Y)$ be two compact metric spaces.
The dilatation of a Lipschitz map $f:X\to Y$ is defined by
\begin{equation}
{\rm dil}(f) :=\sup_{x,y\in X} \frac{d_Y(f(x),f(y))}{d_X(x,y)}.
\end{equation}
The Lipschitz distance (e.g. \cite[Sec. 7.2]{BBI}) between two metric spaces is defined by
\begin{equation} \label{def-Lip-distance}
d_L(X,Y) :=\inf_{f: X\to Y} \log \Big( \max \big\{{\rm dil}(f), {\rm dil}(f^{-1}) \big\}\Big),
\end{equation}
where the infimum is taken over all bi-Lipschitz homeomorphisms $f:X\to Y$.
If there are no bi-Lipschitz homeomorphism from $X$ to $Y$, then their Lipschitz distance is set to be $\infty$.
In general the Lipschitz distance defines a finer topology on compact metric spaces than the Gromov-Hausdorff distance; however, they define the same topology in the class $\M(n,D,K,v_0)$ due to \cite[Thm. 8.25]{Gromov-book} and \cite[Thm. 1]{Kat85}.

\section{Distance coordinates} \label{sec-coordinate}

In this section, we prove Proposition \ref{coordinate-Lip}, the bi-Lipschitz property of a distance coordinate, assuming the bound for the first covariant derivative of the curvature tensor.
This type of results was known in \cite{BKL,BILL} assuming bounds on the higher derivatives of the curvature tensor, where the proof used non-degeneracy of the differential of the exponential map in the distance coordinate, see \cite[Sec. 2.1]{KKL}, together with the inverse function theorem.
In the present paper, we use a different method based on Toponogov's theorem that does not rely on differential structure inspired by \cite{Ivanov1,Ivanov2}.


\begin{lemma} \label{lemma-rL}
Let $(M,g) \in \M(n,D,K,v_0)$ and $U\subset M$ be an open set containing a ball $B(p,r_0)$ of some radius $r_0>0$.
Then there exists a constant $\rL=\rL(n,D,K,v_0,r_0)>0$ such that the following holds for any $\rL$-net $\Gamma$ in $B(p,r_0/2)$.

Let $x,y\in M\setminus U$ satisfy $d(x,y)<{\rm inj}(M)/2$.
Then there exists a uniform constant $\theta_0=\theta_0(n,D,K,v_0,r_0,\rL)>0$ such that there exists a point $z\in \Gamma$ satisfying
$$\big|\angle z xy -\frac{\pi}{2} \big| \geq \theta_0,$$
for some choice of minimizing geodesic $[xz]$ defining the angle.
\end{lemma}

\begin{proof}
We prove by contradiction.
Suppose the statement is not true: then for any $r_k\to 0$, there exists an $r_k$-net $\Gamma_k$ in $B_k(p_k,r_0/2)\subset M_k$ in a sequence of manifolds $M_k \in \M(n,D,K,v_0)$ such that the following holds.
Let $x_k,y_k\in M_k\setminus B_k(p_k,r_0)$ satisfy $d_{M_k}(x_k,y_k)<\inj(M_k)/2$. Then for any $\theta_k\to 0$ and any $z_k\in \Gamma_k$, we have
\begin{equation} \label{assumption-small-angle}
\big|\angle z_k x_k y_k -\frac{\pi}{2} \big| < \theta_k,
\end{equation}
for any choice of minimizing geodesic $[x_k z_k]$ defining the angle. Note that since $d_{M_k}(x_k,y_k)<\inj(M_k)/2$, the minimizing geodesic $[x_k y_k]$ is unique.

By the compactness theorem \cite{A90}, there exists a subsequence of $(M_k,g_k)$ converging to a Riemannian manifold $(M_0,g_0)$ in $C^{1,\alpha}$-topology. 
We write $x_0,y_0,z_0,p_0$ as the limit of $x_k,y_k,z_k,p_k$, respectively.
Under the convergence, the $r_k$-nets $\Gamma_k$ in $B_k(p_k,r_0/2)\subset M_k$ converge to a dense set $\Gamma_0$ in $B(p_0,r_0/2)\subset M_0$.
Moreover, the exponential map of $M_k$ converges uniformly to the exponential map of $M_0$ on compact subsets of the tangent space \cite[Thm. 4.4]{Peters}.
Fixing any choice of minimizing geodesics $[x_k z_k]$,
it follows that the minimizing geodesics $[x_k z_k]$ converge to a geodesic $\gamma_0$ in $M_0$ from $x_0$ to $z_0$, up to a subsequence.
The limit geodesic $\gamma_0$ is minimizing. This can be argued as follows.
Let $F_k: M_k\to M_0$ be the diffeomorphisms defining the $C^{1,\alpha}$-convergence.
Using $(F_k)_* g_k \to g_0$ uniformly,
then for any $\sigma>0$, there exists sufficiently large $k>N$ such that
\begin{eqnarray*}
{\rm Leng}_{g_0} (F_k([x_k z_k])) &\leq& {\rm Leng}_{(F_k)_* g_k} (F_k([x_k z_k]))+ \sigma \\
&=& {\rm Leng}_{g_k} ([x_k z_k])+ \sigma \\
&=& d_{M_k}(x_k, z_k)+\sigma \leq d_{M_0} (x_0,z_0)+2\sigma,
\end{eqnarray*}
where ${\rm Leng}_{g}(\gamma)$ denotes the length of curve $\gamma$ with respect to the Riemannian metric $g$.
Hence,
\begin{eqnarray*}
{\rm Leng}_{g_0}(\gamma_0) \leq \liminf_{k\to \infty} {\rm Leng}_{g_0} (F_k([x_k z_k])) \leq d_{M_0} (x_0,z_0)+2\sigma.
\end{eqnarray*}
Since $\sigma$ is arbitrary, we conclude that $\gamma_0$ is distance minimizing.

Hence, by \eqref{assumption-small-angle} and the continuity of angles, \cite[Thm. 8.41]{AKP} or Lemma \ref{lemma-angle-continuity}, in Riemannian manifolds with two-sided curvature bounds in the sense of Alexandrov, there is some minimizing geodesic $[x_0 z_0]$ for which
\begin{equation} 
 \angle z_0 x_0 y_0  = \lim_{k\to \infty} \angle z_k x_k y_k= \frac{\pi}{2}.
\end{equation}
This holds for all $z_0$ in the dense set $\Gamma_0$ of $B(p_0,r_0/2)\subset M_0$, and for some choice of $[x_0 z_0]$ (that is the limit of some $[x_k z_k]$).
Since $\Gamma_0$ is dense in $B(p_0,r_0/2)$, then for any $q\in B(p_0,r_0/2)$, there exists a sequence $z_{0,j}\in \Gamma_0$ converging to $q$ in $M_0$, so $[x_0 z_{0,j}]\to [x_0 q]$ up to subsequence, and hence the angle converges:
$\angle q x_0 y_0 =\pi/2$ for some choice of $[x_0 q]$.

We have shown that for any $q \in B(p_0,r_0/2)\subset M_0$, there exists some choice of minimizing geodesic $[x_0 q]$ for which
\begin{equation} \label{right-angle}
\angle q x_0 y_0 =\frac{\pi}{2}.
\end{equation}
Let us consider
$$\Omega:= B(p_0,r_0/2) \setminus {\rm Cut}(x_0),$$
where ${\rm Cut(x_0)}$ is the cut locus of $x_0$ in $M_0$.
Since there is only one minimizing geodesic joining $x_0$ and each point in $\Omega$, then \eqref{right-angle} yields 
\begin{equation} \label{direction-zero}
\mathcal{H}^{n-1}\big({\rm dir}(x_0,\Omega) \big)=0,
\end{equation}
where ${\rm dir}(x_0,\Omega)$ denotes the set of the initial directions (in the unit sphere of the tangent space at $x_0$) of all minimizing geodesics from $x_0$ to points of $\Omega$.
Since for Alexandrov spaces ${\rm Cut}(x_0)$ has $n$-dimensional Hausdorff measure zero \cite[Prop. 3.1]{OS}, using \cite[Lemma 2.8]{Ivanov1},
$$\mathcal{H}^{n-1}\big({\rm dir}(x_0,\Omega) \big) \geq C(n,D,K)\frac{{\Vol}(\Omega)}{{\rm diam}(\Omega)}= C(n,D,K) \frac{{\Vol}(B(p_0,r_0/2))}{{\rm diam}(\Omega)} >0.$$
This is a contradiction to \eqref{direction-zero}.
\end{proof}

\begin{lemma}[Angle continuity, Thm 8.41 in \cite{AKP}] \label{lemma-angle-continuity}
Let $X_i$ be a sequence of compact Alexandrov spaces with curvature bounded below by $\kappa$, and $X_i$ converge to $X$ in the Gromov-Hausdorff topology. Suppose that the limit space $X$ has two-sided curvature bounds in the sense of Alexandrov. 
If a sequence of hinges in $X_i$ converge to a hinge in $X$, then the angles of the hinges converge.
\end{lemma}

\begin{proof}
This is due to \cite[Thm. 8.41]{AKP}, using the fact that $q\in {\rm Str}(p)$ implies $p\in {\rm Str}(q)$ under bounded curvature in the sense of Alexandrov, see e.g. \cite[Prop. 3.1]{Ivanov1}.
Recall the definition \cite[Def. 8.10]{AKP} that a point $q$ is called $p$-straight, denoted by $q\in {\rm Str}(p)$, if 
$$\limsup_{r\to q} \frac{d(p,r)-d(p,q)}{d(q,r)}=1.$$
In smooth Riemannian manifolds, this corresponds to the fact that the cut point relation is symmetric: $p\in {\rm Cut}(q)$ implies $q\in {\rm Cut}(p)$.
\end{proof}

\begin{figure}[h]
  \begin{center}
    \includegraphics[width=0.55\linewidth]{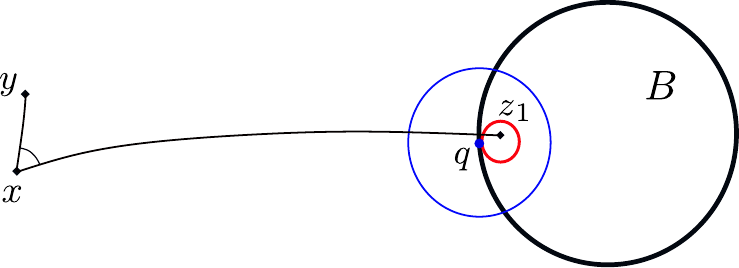}
    \caption{Setting of Lemma \ref{coordinate-local-Lip}. The point $q$ is a nearest point in the closure of $B=B(p,r_0/2)$ from $x$. The blue ball $B(q,\rho_0)$ does not intersect with the cut locus of $x$ for a constant $\rho_0$ determined in Lemma \ref{uniform-cut}. The red ball has radius $\rho_0/4$ centered at some point $q'$. If $\Gamma$ is any $r_L$-net in the red ball with $r_L$ determined in Lemma \ref{lemma-rL}, then there exists a point $z_1\in \Gamma$ such that $|\angle z_1 x y - \pi/2 | \geq \theta_0$.}
    \label{fig_coordinate}
  \end{center}
\end{figure}

To proceed further, we need to identify a subset in $B(p,r_0/2)$ that does not intersect with the cut locus of $x$.
A typical way is to consider a small neighbourhood of a nearest point $q$ in the closure of $B(p,r_0/2)$ from $x$, see e.g. \cite[Sec. 2.1]{KKL}.
Since one can always extend the minimizing geodesic from $x$ to $q$ so that it stays minimizing, $q$ is not a cut point of $x$.
Then there exists a small neighborhood of $q$ that does not contain cut points of $x$.
The following lemma states that there is a uniform choice of radius for the small neighborhood, assuming the bound on the first covariant derivative of the curvature tensor.
It is unclear to us whether this is possible without the additional curvature bound, see \cite[Sec. 2]{Albano} on the need of $C^2$-convergence.

\begin{lemma} \label{uniform-cut}
Let $(M,g) \in \M(n,D,K,K_1,v_0)$ and $U\subset M$ be an open set containing a ball $B(p,r_0)$ of some radius $0<r_0<{\rm inj}(M)/2$. Let $x\in M\setminus U$ be fixed, and
let $q$ be any nearest point in the closure of $B(p,r_0/2)$ from $x$.
Then there exists a constant $\rho_0=\rho_0(n,D,K,K_1,v_0,r_0)$ such that $B(q,\rho_0)$ does not contain any cut point of $x$.
\end{lemma}

\begin{proof}
The proof is a standard compactness argument.
Suppose the claim is not true: then for $x_i\in M_i\setminus B_i(p_i,r_0)$ in a sequence of manifolds $M_i\in \M(n,D,K,K_1,v_0)$ and some nearest point $q_i$ in the closure of $B_i(p_i,r_0/2)$ from $x_i$, there exists a sequence of points $\xi_i \in M_i$ satisfying $d_{M_i}(\xi_i, q_i)\to 0$ such that $\xi_i\in {\rm Cut}(x_i)$.
By compactness \cite{A90,HH}, there is a subsequence of $M_i$ converging to $M_0$ in $C^{2,\alpha}$-topology, under the additional bound for the covariant derivative of the curvature tensor. 
Let $x_0,\xi_0,p_0$ be the limit of $x_i,\xi_i,p_i$ under the convergence, up to subsequence.
Take any minimizing geodesics $[x_i \xi_i]$, and there is a subsequence of $[x_i \xi_i]$ converging to a minimizing geodesic $[x_0 \xi_0]$.
By \cite{BIP, Albano}, the cut points converge to a cut point under $C^2$-convergence of metrics, so $\xi_0\in{\rm Cut}(x_0)$. However, as $\xi_0=q_0=\lim\limits_{i\to \infty} q_i$, and $q_0$ is a nearest point in the closure of $B(p_0,r_0/2)\subset M_0$ from $x_0$, then $q_0\notin {\rm Cut}(x_0)$ because one can always extend $[x_0 q_0]$ to a longer minimizing geodesic. 
\end{proof}

Now we prove the distance coordinate is locally bi-Lipschitz.

\begin{lemma} \label{coordinate-local-Lip}
Let $(M,g) \in \M(n,D,K,K_1,v_0)$ and $U\subset M$ be an open set containing a ball $B(p,r_0)$ of some radius $0<r_0<{\rm inj}(M)/2$. Then there exist uniform constants $\cLip$ and $r_L$ such that the following holds.

Let $\{z_i\}_{i=1}^L$ be any maximal $r_L$-separated set in $B(p,r_0/2)$.
Define the distance coordinate $\Phi: M\to \mathbb{R}^L$ with respect to the set by
\begin{equation*} \label{def-distance-coordinate}
\Phi_L(x) :=\big(d(x,z_1),\cdots, d(x,z_L) \big).
\end{equation*}
Let $x,y\in M\setminus U$ be arbitrary satisfying $d(x,y)< \inj(M)/2$.
If $d(x,y)<\cLip$, then 
$$C^{-1} \leq \frac{|\Phi_L(x) -\Phi_L(y)|}{d(x,y)} \leq C,$$
where the bi-Lipschitz constant $C$ is uniform in the class of manifolds.
\end{lemma}

\begin{proof}
We take $q$ to be a nearest point in the closure of $B(p,r_0/2)$ from $x$.
By Lemma \ref{uniform-cut}, there exists $\rho_0>0$ such that $B(q,\rho_0)\cap{\rm Cut}(x)=\emptyset$.
There is a ball $B(q', \rho_0/4)$ centered at some point $q'$ of radius $\rho_0/4$ contained in $B(q,\rho_0)\cap B(p,r_0/2)$, see Figure \ref{fig_coordinate}.
Then we apply Lemma \ref{lemma-rL} to the ball $B(q',\rho_0/4)$ and determine the constant $r_L$, say $r_L<\rho_0/20$.
Let $\Gamma=\{z_i\}_{i=1}^L$ be any maximal $r_L$-separated set (hence an $r_L$-net) in $B=B(p,r_0/2)$, so the total number of points in the set is uniformly bounded:
\begin{equation} \label{bound-L}
L\leq C(n,D,K) r_L^{-n}.
\end{equation}
It follows that the $r_L$-neighborhood of $\Gamma\cap B(q', \rho_0/8)$ covers $B(q', \rho_0/8-r_L)$.
Hence by Lemma \ref{lemma-rL}, there exists a uniform constant $\theta_0>0$ and a point, say $z_1\in \Gamma\cap B(q',\rho_0/8)$,
such that $|\angle z_1 x y -\frac{\pi}{2}|\geq \theta_0$ for some choice of minimizing geodesic $[x z_1]$ defining the angle, that is,
\begin{equation} \label{bound-theta}
|\cos \angle z_1 x y| \geq \sin\theta_0>0.
\end{equation}
Since $B(q',\rho_0/4)\cap {\rm Cut}(x)=\emptyset$, 
any (unique) minimizing geodesic from $x$ to points in $B(q',\rho_0/8)$ can be extended to a longer minimizing geodesic by at least length $\rho_0/10$.
Then we apply Corollary 3.2 in \cite{Ivanov1},
\begin{equation}
\Big| |z_1 x| - |z_1 y| -|xy| \cos\angle z_1 x y \Big| \leq C(n,D,K,v_0,r_0,\rho_0) |xy|^2.
\end{equation}
Then by \eqref{bound-theta},
$$\big| |z_1 x| - |z_1 y| \big| \geq |xy| \sin\theta_0 -C(n,D,K,v_0,r_0,\rho_0)|xy|^2.$$
When $|xy|$ is small depending on $n,D,K,v_0,\theta_0,r_0,\rho_0$, we see that
\begin{equation}
\big| |z_1 x| - |z_1 y| \big| \geq C(\theta_0)|xy|.
\end{equation}
Hence,
$$|\Phi_L(x)-\Phi_L(y)| \geq \big| |z_1 x| - |z_1 y| \big| \geq C(\theta_0) |xy|.$$
On the other hand, by triangle inequality,
$$|\Phi_L(x)-\Phi_L(y)| \leq \sqrt{L}|xy|,$$
where $L$ is uniformly bounded by \eqref{bound-L}.
\end{proof}

Using Lemma \ref{coordinate-local-Lip} and a standard compactness argument, one can show that the distance coordinate is globally bi-Lipschitz.

\begin{proposition} \label{coordinate-Lip}
Let $(M,g) \in \M(n,D,K,K_1,v_0)$ and $U\subset M$ be an open set containing a ball $B(p,r_0)$ of some radius $0<r_0<{\rm inj}(M)/2$. Then there exists a uniform constant $r_L>0$ such that the following holds.

Let $\{z_i\}_{i=1}^L$ be any maximal $r_L$-separated set in $B(p,r_0/2)$.
Then the distance coordinate $\Phi: M\to \mathbb{R}^L$ defined by
\begin{equation} \label{def-distance-coordinate}
\Phi_L(x) :=\big(d(x,z_1),\cdots, d(x,z_L) \big) 
\end{equation}
is bi-Lipschitz on $M\setminus U$. The bi-Lipschitz constant and $r_L$ depend on $n,D,K,K_1,v_0,r_0$.
\end{proposition}

\begin{proof}
Since $M$ is compact, local Lipschitz implies global Lipschitz.
Suppose bi-Lipschitz does not hold: then there exist a sequence of manifolds $M_k\in \M(n,D,K,K_1,v_0)$, some maximal $1/k$-separated set in $B_k(p_k,r_0/2)\subset M_k$, and some points $x_k,y_k \in M_k\setminus B_k(p_k,r_0)$ such that
$$\frac{|\Phi_{L,k}(x_k)-\Phi_{L,k}(y_k)|}{d_{M_k}(x_k,y_k)}\to 0, \quad \textrm{ as }k\to \infty,$$
where $\Phi_{L,k}$ is defined with respect to the maximal $1/k$-separated set.
Up to subsequences, we may assume that $M_k\to M_0$, $x_k\to x_0$, $y_k\to y_0$, $p_k\to p_0$ and the maximal $1/k$-separated sets converge to a dense set in $B(p_0, r_0/2)$, as $k\to \infty$.
This implies that $d_{M_0}(x_0,\xi)=d_{M_0}(y_0,\xi)$ for any $\xi\in B(p_0,r_0/2)$,
which yields $x_0=y_0$, see e.g. \cite[Lemma 3.30]{KKL}. Hence, when $k$ is sufficiently large, we have $d_{M_k}(x_k,y_k)<\min\{\cLip,\inj(M_k)/2 \}$ and $1/k<r_L$ for the constants $\cLip,r_L$ determined in Lemma \ref{coordinate-local-Lip}, leading to a contradiction to local bi-Lipschitz.
\end{proof}

\section{Reconstruction of metric structure} \label{sec-metric}

In this section, we reconstruct the metric structure of $(M,g)$ from an approximation of the spectral data on an open subset $U\subset M$. We mainly follow the method in \cite{BKL,BILL} but need to make modifications for the purpose of reconstructing the potential in Section \ref{sec-potential}. 
First, the quantitative unique continuation gives an approximate reconstruction of the Fourier coefficients of functions restricted on the domain of influence, which in particular approximates the $L^2$-norm of the first eigenfunction restricted on the domain of influence. Next, we use the approximation of the $L^2$-norm to perform a slicing procedure on the manifold and construct a Hausdorff approximation of the interior distance functions.
Given an open subset $U\subset M$, recall that the interior distance function $r_x: U\to \R$ corresponding to a point $x\in M$ is defined by $r_x(z)=d(x,z)$ for $z\in U$, see e.g. \cite[Sec. 3.7]{KKL}.
We denote the set of interior distance functions on $U$ by
$$\mathcal{R}_U (M) = \{r_x \in L^{\infty}(U): x\in M\}.$$

Suppose that $U$ contains a ball $B(p,2r_0)$ and the Riemannian metric on $U$ is given. 
Without loss of generality, we assume that
$$0<r_0<\frac{i_0}{4} ,$$
where $i_0=i_0(n,D,K,v_0)$ is the lower bound for the injectivity radius in \eqref{inj}.
Let $0<\e<r_0/64$ be fixed in this section.
We choose open subsets of $B(p,r_0)$ in the following way. 
Let $\{z_k\}_{k=1}^N$ be a maximal $\e/2$-separated set in $B(p,r_0/2)$.
Let $\{U_k\}_{k=1}^N$ be disjoint open subsets containing $z_k$ in $B(p,r_0)$ satisfying
\begin{equation} \label{partition}
B(p, \frac{r_0}{2}) \subset \bigcup_{k=1}^N \overline{U_k} \subset B(p,\frac{3r_0}{4}), \quad \textrm{diam}(U_k)\leq \e.
\end{equation}
Without loss of generality, assume that every $U_k$ contains a ball of radius $\e/4$. 
Note that with the choice of $U_k$ above, the total number $N$ is bounded above:
\begin{equation} \label{bound-N}
N\leq C(n,D,K) \e^{-n}.
\end{equation}
Let $\alpha=(\alpha_1,\cdots,\alpha_N)$ with $\alpha_k\in [\e, D]\cup \{0\}$ be a multi-index. We define the domain of influence associated with $\alpha$ by
\begin{equation}\label{Xalpha}
M_{\alpha}:=\bigcup_{k=1}^N M(U_k, \alpha_k)=\bigcup_{k=1}^N \big\{ x \in M: d(x,U_k) < \alpha_k \big\}.
\end{equation}

\subsection{Approximate reconstruction of Fourier coefficients}

In this subsection we approximate the Fourier coefficients of $\chi_{M_{\alpha}}u$, where $\chi$ is the characteristic function.
Equipped with the quantitative unique continuation result, Theorem \ref{uc-Y}, the proof of this subsection is essentially identical to that of \cite{BKL,BILL}.
We shall only explain the part that is specific to our setting.

Let $q$ be a potential function on $M$ satisfying \eqref{bound-q-positive}.
Let $\mathcal{V}_J={\rm span}\{\phi_1,\cdots,\phi_J\}\subset C^{\infty}(M)$ be the span of the first $J$ eigenfunctions, where $\{\phi_j\}_{j=1}^{\infty}$ is a complete family of $L^2$-orthonormalized eigenfunctions of $-\Delta_g +q$ on $M$.
Given $E>1$ and a small parameter $\e_1>0$,
we consider the set
\begin{equation} \label{def-U}
\mathcal{U}(J,E,\e_1) :=\bigcap_{k=1}^N \, \Big\{v\in \mathcal{V}_J: \, \|v\|_{1;q}\leq E, \quad \|W(v)\|_{H^1(U_k\times [-\alpha_k,\alpha_k])} \leq \e_1 \Big\},
\end{equation}
where $W(v)$ is the solution to the wave equation with the initial value $v$:
\begin{equation}
\partial_t^2 W -\Delta_g W + qW=0\;\; \textrm{ in }M\times \R,\quad \quad W|_{t=0}=v, \quad \partial_t W|_{t=0}=0.
\end{equation}
In \eqref{def-U} the quantity $\|v\|_{1;q}$ is defined by
\begin{equation}
\|v\|_{1;q}^2 := \int_M |\nabla v|^2 + \int_M qv^2.
\end{equation}
Observe that this quantity $\|v\|_{1;q}$ can be tested using the spectral data. 
Namely, if $v=\sum_{j=1}^J v_j \phi_j$, then
\begin{equation}
\|v\|_{1;q}^2= \langle (-\Delta_g +q)v , v\rangle_{L^2(M)} = \sum_{j=1}^J \lambda_j v_j^2.
\end{equation}
In particular, under the conditions \eqref{bound-q-positive}, this quantity is equivalent to the $H^1$-norm:
\begin{equation}
\Cq^{-\frac12} \|v\|_{H^1(M)} \leq \|v\|_{1;q} \leq \Cq^{\frac12} \|v\|_{H^1(M)}.
\end{equation}
For the second condition in \eqref{def-U}, it can also be tested using the spectral data on $U$, since 
\begin{equation}
W(v)(x,t) =\sum_{j=1}^J v_j \cos(\sqrt{\lambda_j}t) \phi_j(x),
\end{equation}
which is a known function on $(U,g|_U)$. Note that we have used \eqref{bound-lambda1} under the conditions \eqref{bound-q-positive}. 
Thus, we have shown that the conditions in \eqref{def-U} can be tested by the spectral data.

\smallskip
Now suppose that we only know a $\delta$-approximation $\{\lambda_j^a,\phi_j^a|_U\}$ of the spectral data $\{\lambda_j,\phi_j|_U\}$ on $U$ in the sense of Definition \ref{def-data-close}.
If $v=\sum_{j=1}^J v_j \phi_j$ satisfies $\|v\|_{L^2}\leq 1$, one can verify that
\begin{equation}
\Big| \|v\|^2_{1;q} - \sum_{j=1}^J \lambda_j^a v_j^2 \Big| <\delta,
\end{equation}
and
\begin{equation}
\Big\| W(v) - \sum_{j=1}^J v_j \cos(\sqrt{\lambda_j^a}t) \phi_j^a \Big\|_{H^1(U\times [-D,D])} < C(D,\Cq,\Vol(M)) J \lambda_J^{\frac12} \delta.
\end{equation}
Therefore, the set $\mathcal{U}$ needs to be replaced with
\begin{eqnarray} \label{def-Ua}
\mathcal{U}^a :=\bigcap_{k=1}^N \, \Big\{v=\sum_{j=1}^J v_j \phi_j: \, &&\sum_{j=1}^J v_j^2\leq 1,\quad \sum_{j=1}^J \lambda_j^a v_j^2 \leq E^2+\delta,  \nonumber \\
&& \Big\|\sum_{j=1}^J v_j \cos(\sqrt{\lambda_j^a}t) \phi_j^a \Big\|_{H^1 (U_k\times [-\alpha_k,\alpha_k])} \leq \e_1+J\lambda_J^{\frac12}\delta \; \Big\},
\end{eqnarray}
so that the conditions in $\mathcal{U}^a$ can be tested using an approximation of the spectral data on $U$.
Suppose that we are given the first $J$ Fourier coefficients of a function $u$.
The following minimization problem of finding the Fourier coefficients $\{v_j\}_{j=1}^J$ of a function in $\mathcal{U}^a$ that minimizes
\begin{equation}
 \min_{w\in \mathcal{U}^a} \|w-u\|_{L^2(M)},
\end{equation}
is solvable, since \eqref{def-Ua} defines a closed condition for the coefficients $\{v_j\}_{j=1}^J \in \R^J$.


\smallskip
From here, using Theorem \ref{uc-Y} and following the proof in \cite[Sec. 4]{BILL} or \cite{BKL}, we have the following result.

\begin{proposition} \label{prop-Fourier}
Let $u\in H^2(M)$ satisfy $\|u\|_{L^2(M)}=1$ and $\|u\|_{H^2(M)}\leq \Lambda$. 
Let $\alpha=(\alpha_1,\cdots,\alpha_N)$, $\alpha_k\in [\e,D]\cup \{0\}$ be given. 
Then for any $\sigma>0$, there exists sufficiently large $J$ such that the following holds.

There exists sufficiently small $\delta=\delta(\sigma,\e) \leq J^{-1}$ such that
by knowing a $\delta$-approximation $\{\lambda_j^a,\phi_j^a|_U\}$ of the spectral data of $-\Delta_g+q$ on $U$ and the first $J$ Fourier coefficients of $u$, we can find $\{b_j\}_{j=1}^J$ and $u^a=\sum_{j=1}^J  b_j\phi_j$ such that
$$\|u^a- \chi_{M_{\alpha}} u\|_{L^2(M)}\leq \sigma.$$
In the above the Fourier coefficients are with respect to a choice of spectral data $\{\lambda_j,\phi_j|_U\}$ such that Definition \ref{def-data-close} is satisfied.
\end{proposition}

We apply Proposition \ref{prop-Fourier} to the (positive) first eigenfunction $\phi_1$.

\begin{lemma} \label{lemma-appnorm}
Let $\alpha=(\alpha_1,\cdots,\alpha_N)$, $\alpha_k\in [\e,D]\cup \{0\}$ be given. 
Then for any $\sigma\in (0,1)$, there exists sufficiently small $\delta=\delta(\sigma,\e)$, such that 
by knowing a $\delta$-approximation $\{\lambda_j^a,\phi_j^a|_U\}$ of the spectral data on $U$,
we can compute a number $\appnorm(M_{\alpha})$ satisfying
$$\Big| \appnorm(M_{\alpha}) - \int_{M_{\alpha}} \phi_1^2 \,\Big| \leq \sigma.$$
\end{lemma}

\begin{proof}
Taking $u=\phi_1>0$, Proposition \ref{prop-Fourier} says that for any multi-index $\alpha$, there exists sufficiently large $J$ such that we can find numbers $\{b_j\}_{j=1}^J$ and $u^a=\sum b_j \phi_j$ satisfying
$$\|u^a- \chi_{M_{\alpha}} \phi_1 \|_{L^2(M)}\leq \frac13 \sigma.$$
By the triangle inequality,
$$\Big| \|u^a\|_{L^2}-(\int_{M_{\alpha}} \phi_1^2)^{1/2} \Big|\leq \frac13 \sigma.$$
Then defining 
$$ \appnorm(M_{\alpha}):= \|u^a\|_{L^2}^2=\sum_{j=1}^J b_j^2$$
satisfying the claim, since $\|\phi_1\|_{L^2(M)}=1$.
\end{proof}

\subsection{Approximate reconstruction of metric} \label{subsec-slicing}
Next, we use Lemma \ref{lemma-appnorm} to define a slicing procedure on the manifold similar to \cite{BKL,BILL}.
However, for the purpose of recovering the potential later in Section \ref{sec-potential}, we need to modify the procedure so that different slices are disjoint.

Given any fixed $\e>0$,
let $\{z_k\}_{k=1}^N$ be a maximal $\e/2$-separated set (hence an $\e/2$-net) in $B(p,r_0/2)$, such that $\{z_k\}_{k=1}^L$ for $L<N$ is a maximal $r_L$-separated set in $B(p,r_0/2)$, where $r_L$ is the uniform constant determined in Proposition \ref{coordinate-Lip}.
The number $L$ is bounded by a uniform constant \eqref{bound-L}.
Let $\{U_k\}_{k=1}^N$ be the partition of $U$ as defined in \eqref{partition}.

Given a multi-indice $\beta=(\beta_1,\cdots,\beta_N)$, $\beta_k\in \N$,
we define a ``slice" of the manifold
\begin{equation} \label{def-rbetastar}
M_{\beta}^{\ast} := \bigcap_{k: \, \beta_k>0} \Big\{ x\in M: d(x,U_k)\in [\beta_k \e -\e,\beta_k\e)\Big\}.
\end{equation}
Note that $M_{\beta}^{\ast}$ for different $\beta$ are disjoint by definition, which is later necessary for reconstructing the potential. We also need to define a modified slicing
\begin{equation} \label{def-rbetae}
M_{\beta}^{\e} := \bigcap_{k: \, \beta_k>0} \Big\{ x\in M: d(x,U_k)\in [\beta_k \e -\e-\e^2,\beta_k\e+\e^2)\Big\},
\end{equation}
and a modified multi-index
\begin{equation}
\beta\langle l \rangle:= (\beta_1,\cdots,\beta_L, 0,\cdots,0, \beta_l,0,\cdots,0),\quad \textrm{ for }l\in \{L+1,\cdots,N \}.
\end{equation}

\begin{lemma} \label{norm-beta}
Under the setting of Lemma \ref{lemma-appnorm}, for any $\sigma\in (0,1)$, there exists sufficiently small $\delta=\delta(\sigma,\e)$, such that 
by knowing a $\delta$-approximation $\{\lambda_j^a,\phi_j^a|_U\}$ of the spectral data of $-\Delta_g +q$ on $U$,
we can compute a number $\appnorm(M_{\beta \langle l \rangle}^{\e})$ satisfying
$$\Big| \appnorm(M_{\beta \langle l \rangle}^{\e}) - \int_{M_{\beta \langle l \rangle}^{\e}} \phi_1^2 \,\Big| \leq 2^{L+1}\sigma, \quad \textrm{ for any }l\in \{L+1,\cdots,N \}.$$
\end{lemma}

\begin{proof}
The proof is identical to \cite[Lemma 9]{BKL} or \cite[Lemma 5.4]{BILL}.
The idea is to write the integral over $M_{\beta \langle l \rangle}^{\e}$ as the sum of integrals over sets of the form $M_{\alpha}$ for some $\alpha$. The number of terms in the sum is bounded by $2^{L+1}$ since only $(L+1)$-number of $U_k$ appear in $M_{\beta \langle l \rangle}^{\e}$. Then the claim directly follows from Lemma \ref{lemma-appnorm}.
\end{proof}

\begin{definition} \label{def-rbeta}
We define a set of functions on $B=B(p,r_0/2)$ in the following way.
\begin{itemize}
\item[(1)] For a multi-index $\beta$ satisfying $\beta_k \e\geq r_0/8$ for all $k=1,\cdots,N$, we associate to it a piecewise constant function $r_{\beta}\in L^{\infty}(B)$ defined by
$$r_{\beta}(z):=\beta_k \e, \quad \textrm{ if }z\in U_k ,$$
if the following holds:
\begin{equation} \label{beta-cri}
\appnorm (M_{\beta \langle l \rangle }^{\e})\geq c_{\ast}\e^{2n}, \quad \textrm{ for all }l=L+1,\cdots, N,
\end{equation}
where $c_{\ast}$ is a uniform constant to be determined later in \eqref{cstar}.
We test all multi-indices $\beta$ up to $\beta_k\leq 1+D/\e$ for each $k$, and denote the set of all functions chosen in this way by $\mathcal{R}_{out}^*$.

\item[(2)] We take an arbitrary $\e$-net $\Gamma_{\e}$ containing the base point $p$ in $B(p,r_0)$. Define
$$\mathcal{R}_{in}^* := \{r_{\xi} \in \mathcal{R}_{B}(M): \xi \in \Gamma_{\e} \},$$
where $r_{\xi}$ is the interior distance function on $B$ corresponding to the point $\xi$.
\end{itemize}
We denote
$$\mathcal{R}^* := \mathcal{R}_{in}^* \cup \mathcal{R}_{out}^*.$$
\end{definition}

Note that for $r_{\beta}\in \mathcal{R}_{out}^*$, it could happen that $M_{\beta}^*$ is empty, or $M_{\beta \langle l \rangle}^*$ can be empty for every $l$. However, we will show that the slightly larger $M_{\beta \langle l \rangle}^{\e}$ must be nonempty for all $l$, with suitable chosen uniform constant $c_*$. 
The elements in $\mathcal{R}_{in}^*$ are determined by the given Riemannian metric on $B(p,2r_0)\subset U$, because the minimizing geodesic between any pair of points in $B(p,r_0)$ lies in $B(p,2r_0)$.

\smallskip
To prove that the finite set $\mathcal{R}^*$ approximates the interior distance functions $\mathcal{R}_B(M)$, we rely on the fact that the first eigenfunction $\phi_1$ is uniformly bounded below (Lemma \ref{lemma-eigen-positive}).

\begin{lemma} \label{Hausdorff-close}
Let $(M,g) \in \M(n,D,K,K_1,v_0)$ and $U\subset M$ be an open set containing a ball $B(p,r_0)$ of radius $0<r_0<\inj(M)/2$.
Let $q$ be a potential function on $M$ satisfying \eqref{bound-q-positive}.
Then for any $\e>0$, there exists $\delta_1>0$, such that any $\delta_1$-approximation $\{\lambda_j^a,\phi_j^a|_U\}$ of the spectral data of $-\Delta_g +q$ on $U$ determines a finite set $\mathcal{R}^*$ satisfying
$$d_H(\mathcal{R}^*, \mathcal{R}_B (M)) \leq \CHau \e,$$
where $B=B(p,r_0/2)$ and $d_H$ denotes the Hausdorff distance between subsets of $L^{\infty}(B)$.
The constant $\CHau$ depends on $n,D,K,K_1,v_0,r_0$, and the choice of $\delta_1$ depends on the same set of parameters and additionally on $\e,\Cq$.
\end{lemma}

\begin{proof}
Let $x\in M$ be arbitrary.
If $x\in B(p,r_0)$, then there exists $\xi \in \Gamma_{\e}$ such that $d(x,\xi)<\e$. Hence,
$$\|r_x-r_{\xi}\|_{L^{\infty}(B)}\leq d(x,\xi)<\e.$$
Namely, there exists $r_{\xi}\in \mathcal{R}_{in}^*$ such that $\|r_x-r_{\xi}\|_{L^{\infty}(B)}<\e$.
Now suppose that $x\in M\setminus B(p,7r_0/8)$.
Then there exists a multi-index $\beta$, with $\beta_k>1$ for all $k$, such that $x\in M_{\beta}^*$.
Since $d(x,\cup_k U_k)\geq r_0/8$ by our choice of $U_k$, it must satisfy that $\beta_k \e\geq r_0/8$ for all $k=1,\cdots,N$.
Consider the ball $B(x,\e^2)$, and observe that $B(x,\e^2/2)\subset M_{\beta}^{\e}$.
By the choice of $\beta$, we have $\|r_x -r_{\beta}\|_{L^{\infty}(B)}\leq 3\e+\e^2<4\e$.
We need to show that this choice of $\beta$ satisfies \eqref{beta-cri} and hence belongs to $\mathcal{R}_{out}^*$.
To this end, by Lemma \ref{lemma-eigen-positive},
$$\int_{M_{\beta}^{\e}} \phi_1^2 \geq \int_{B(x,\e^2/2)} \phi_1^2 \geq  \ceigen^2 c(n) \e^{2n},$$
where $c(n)$ is a constant depending only on $n$.
Since clearly $M_{\beta}^{\e} \subset M_{\beta \langle l \rangle}^{\e}$ for all $l$, then every $M_{\beta \langle l \rangle}^{\e}$ contains $B(x,\e^2/2)$, so the above holds for every $\beta \langle l \rangle$.
In Lemma \ref{norm-beta}, we choose
\begin{equation} \label{sigma1}
\sigma=\frac14 2^{-(L+1)}\ceigen^2 c(n) \e^{2n},
\end{equation}
which determines the choice of sufficiently small $\delta$: this is the choice of $\delta_1$.
Then by Lemma \ref{norm-beta},
$$\appnorm(M_{\beta \langle l \rangle}^{\e})>  \frac12 \ceigen^2 c(n) \e^{2n}.$$
This gives the choice of $c_*$ in Definition \ref{def-rbeta}:
\begin{equation} \label{cstar}
c_*=\frac12 \ceigen^2 c(n).
\end{equation}

For the other direction, let $r^* \in \mathcal{R}^*$ be arbitrary.
If $r^*=r_{\xi}\in \mathcal{R}_{in}^*$ for some $\xi \in B(p,r_0)$, then $r_{\xi}$ immediately satisfies the claim.
Suppose that $r^*=r_{\beta}\in \mathcal{R}_{out}^*$.
By Definition \ref{def-rbeta}, Lemma \ref{norm-beta} and the choice of $\sigma$ in \eqref{sigma1}, we have
$$\int_{M_{\beta  \langle l \rangle }^{\e}} \phi_1^2 \geq c_* \e^{2n}-2^{L+1} \sigma \geq \frac14 \ceigen^2 c(n) \e^{2n}>0.$$
In particular, $M_{\beta  \langle l \rangle }^{\e}\neq \emptyset$ for all $l=L+1,\cdots,N$.
For each $l$, we take a point $y_l\in M_{\beta  \langle l \rangle }^{\e}$.
Then by definition, 
$$\|r_{y_l} - r_{\beta}\|_{L^{\infty}(U_1\cup \cdots \cup U_L \cup U_l)}\leq 3\e+\e^2<4\e.$$
Moreover, by the condition $\beta_k \e\geq r_0/8$ for elements in $\mathcal{R}_{out}^*$,
then $d(y_l, z_k)\geq \beta_k \e-3\e>r_0/16$ for sufficiently small $\e<r_0/64$, for all $k=1,\cdots,L$.
Since $\{z_k\}_{k=1}^L$ is an $r_L$-net in $B=B(p,r_0/2)$, choosing $r_L<r_0/32$, we have
$d(y_l,B)>r_0/32$.
Thus, Proposition \ref{coordinate-Lip} is applicable,
$$d(y_l,y_k) \leq C|\Phi_L(y_l)-\Phi_L(y_k)|\leq 5C\sqrt{L} \e, \quad \textrm{ for any }l,k\in \{L+1,\cdots,N\},$$
where $C$ is the uniform bi-Lipschitz constant depending on $n,D,K,K_1,v_0,r_0$.
Now fix any $y_l$, for all $k=L+1,\cdots,N$,
$$\|r_{y_l}-r_{\beta}\|_{L^{\infty}(U_k)} \leq \|r_{y_l}-r_{y_k}\|_{L^{\infty}(U_k)} +\|r_{y_k}-r_{\beta}\|_{L^{\infty}(U_k)} < (5C\sqrt{L}+4)\e. $$
Hence, ranging $k$ over $\{L+1,\cdots,N\}$, we have
\begin{equation*}
\|r_{y_l}-r_{\beta}\|_{L^{\infty}(B)} < (5C\sqrt{L}+4)\e.  \qedhere
\end{equation*}
\end{proof}

\begin{definition} \label{def-corres}
For an element $r_{\beta}\in \mathcal{R}_{out}^*$, we assign to it a point $x\in M_{\beta \langle l \rangle}^{\e}$ for some $l$, and call $x$ a corresponding point to $r_{\beta}$. 
The existence of corresponding points is guaranteed by the proof of Lemma \ref{Hausdorff-close}, and it satisfies $d(x,B)>r_0/32$ and
$$\|r_x-r_{\beta}\|_{L^{\infty}(B)} \leq \CHau \e.$$
For an element $r_{\xi}\in \mathcal{R}_{in}^*$, we assign the point $\xi \in B(p,r_0)$ as the corresponding point to $r_{\xi}$.
We fix one corresponding point to each element $r^* \in \mathcal{R}^*$.
We denote this set of points by
\begin{equation}
X:=\{x_1,\cdots,x_I\}.
\end{equation}
In particular the base point $p\in X$, since we have chosen $r_p\in \mathcal{R}_{in}^*$ in Definition \ref{def-rbeta}(2).
\end{definition}

\begin{remark}
Note that for $r_{\beta}\in \mathcal{R}_{out}^*$, it could happen that $M_{\beta}^*$ is empty, or $M_{\beta \langle l \rangle}^*$ can be empty for every $l$.
Moreover, the chosen corresponding point $x$ does not necessarily belong to $M_{\beta}^*$ or $M_{\beta \langle l \rangle}^*$ for any $l$.
However, thanks to the bi-Lipschitz property of distance coordinate (Proposition \ref{coordinate-Lip}), it does happen that $M_{\beta}^*$ lies in the $\CHau \e$-neighborhood of any chosen corresponding point, see Lemma \ref{lemma-proximity}.
\end{remark}



Now with Lemma \ref{Hausdorff-close}, applying \cite[Thm. 1.2]{FILLN} gives the following result.

\begin{proposition}
\label{prop-metric}
Let $(M,g) \in \M(n,D,K,K_1,v_0)$ and $U\subset M$ be an open set containing a ball $B(p,2r_0)$ of radius $0<r_0<{\rm inj}(M)/4$.
Let $q$ be a potential function on $M$ satisfying \eqref{bound-q-positive}.
Then for any $\e>0$, there exists $\delta_1>0$, such that 
any $\delta_1$-approximation $\{\lambda_j^a,\phi_j^a |_U\}$ of the spectral data of $-\Delta_g +q$ on $U$ determine the numbers $\h{d}_{ik}$, for $i,k\in \{1,2,\cdots,I\}$, such that
$$|\h{d}_{ik}-d(x_i,x_k)|\leq \Cd \e^{\frac18}, \quad \textrm{for any }x_i,x_k\in X,$$
where $X$ is a choice of corresponding points in the sense of Definition \ref{def-corres}.
Furthermore, the given data determine a Riemannian manifold $(\widehat{M},\widehat{g})$ that is diffeomorphic to $M$ via a bi-Lipschitz diffeomorphism $F:M\to \widehat{M}$ satisfying
$$1-C_1 \e^{\frac{1}{12}} \leq \frac{d_{\widehat{M}} (F(x), F(y))}{d_M (x,y)} \leq 1+C_1 \e^{\frac{1}{12}}, \quad \textrm{ for any }x,y\in M,$$
for some uniform constant $C_1>1$.
In particular, this gives an estimate for the Lipschitz distance:
$$d_{L}\big( (M,g), (\widehat{M},\widehat{g}) \big)\leq C_1\e^{\frac{1}{12}}.$$
The constants $C_1,\Cd$ depend on $n,D,K,K_1,v_0,r_0$, and
the choice of $\delta_1$ depends on the same set of parameters and additionally on $\e,\Cq$.
\end{proposition}

This proves the first part of Theorem \ref{main1}.
Note that the set $X$ is not determined by the given spectral data: we only know its existence.


\section{Reconstruction of the potential}
\label{sec-potential}

In this section, we reconstruct the potential function from an approximation of the spectral data on an open subset $U\subset M$, with the help of the approximation of metric structure that we established in Proposition \ref{prop-metric}.
The basic idea is to use the approximation of the $L^2$-norm of the first eigenfunction $\phi_1$ over sets of the form $M_{\beta}^*$ to define a discrete operator that approximates the Laplacian on the set $X$ of corresponding points.
We will prove that the given spectral data determine the discrete operator defined by Definition \ref{def-graph-Lap}, and this operator on a discrete function approximates the Laplacian of $\phi_1$ on $X$ with a uniform estimate, Proposition \ref{appro-Lap}.
To avoid bringing in a third logarithm in the final result, we define a partition of $M\setminus U$ using only a fixed number of coordinates in multi-indices $\beta$.
This is made possible by Proposition \ref{coordinate-Lip}.

Suppose that $U$ contains a ball $B(p,2r_0)$. 
Let $\e>0$ be fixed. 
Let $\{z_k\}_{k=1}^N$ be a maximal $\e/2$-separated set (hence an $\e/2$-net) in $B=B(p,r_0/2)$, such that $\{z_k\}_{k=1}^L$ is a maximal $r_L$-separated set in $B$ for the uniform constant $r_L$ determined in Proposition \ref{coordinate-Lip}, where the number $L<N$ is bounded by a uniform constant \eqref{bound-L}.
We assume that $\e,r_L<r_0/64$.

\begin{definition}
We define the set of the first $L$ indices of all multi-indices corresponding to elements in $\mathcal{R}_{out}^*$ as
$$\mathcal{A} := \big\{(\beta_1,\cdots,\beta_L,0,\cdots,0)  : r_{\beta} \in \mathcal{R}_{out}^* \big\},$$
where $\mathcal{R}_{out}^*$ is defined in Definition \ref{def-rbeta}(1).

\end{definition}

\begin{lemma} \label{lemma-proximity}
We have the following properties.
\begin{itemize}
\item[(1)] The sets $M_{\tau}^*$ are disjoint for different $\tau\in \mathcal{A}$, and their union covers $M\setminus B(p,r_0)$.

\item[(2)] For any $\tau\in \mathcal{A}$, there exists a point $x_{\tau} \in X$, satisfying $d(x_{\tau},B)>r_0/32$, such that $M_{\tau}^*$ is contained in the $\CHau \e$-neighborhood of $x_{\tau}$, where the set $X$ has been chosen in Definition \ref{def-corres}. In particular, we have ${\rm diam}(M_{\tau}^*) < 2 \CHau \e$.
\end{itemize}
Note that it is possible that $M_{\tau}^*$ may be empty or disconnected, or $x_{\tau}$ may not belong to $M_{\tau}^*$.
\end{lemma}

\begin{proof}
(1) Recall the definition \eqref{def-rbetastar}: for $\tau=(\tau_1,\cdots,\tau_L,0,\cdots,0) \in \mathcal{A}$, $\tau_k\in \N$,
\begin{equation} \label{Mtau}
M_{\tau}^{\ast} := \bigcap_{k=1}^L \Big\{ x\in M: d(x,U_k)\in [\tau_k \e -\e,\tau_k\e)\Big\},
\end{equation}
since $\tau_k\e\geq r_0/8$ for all $k=1,\cdots,L$, by Definition \ref{def-rbeta}.
It is clear that $M_{\tau}^{\ast}$ are disjoint for different $\tau$.
Let $y\in M\setminus B(p,r_0)$ be arbitrary. By the first part of the proof of Lemma \ref{Hausdorff-close}, there exists a multi-index $\beta$ satisfying $y\in M_{\beta}^*$ such that $r_{\beta}\in \mathcal{R}_{out}^*$.
We take $\tau=(\beta_1,\cdots, \beta_L,0,\cdots,0)$ which belongs to $\mathcal{A}$. Since $M_{\beta}^* \subset M_{\tau}^*$, then $y\in M_{\tau}^*$.

\smallskip
(2) Let $\tau\in \mathcal{A}$ be arbitrary. By definition, there exists a multi-index $\beta$ corresponding to $r_{\beta} \in \mathcal{R}_{out}^*$ such that $\tau=(\beta_1,\cdots, \beta_L,0,\cdots,0)$.
By the second part of the proof of Lemma \ref{Hausdorff-close}, there exists a corresponding point $x_{\tau}\in M_{\beta\langle l \rangle}^{\e}\neq \emptyset$ for some $l\in \{L+1,\cdots,N\}$, satisfying $d(x_{\tau},B)>r_0/32$.
Assuming the set $X$ has already been chosen in Definition \ref{def-corres}, we can simply pick $x_{\tau}$ to be the point in $X$ corresponding to $r_{\beta}$.
Then by definition \eqref{def-rbetae} and triangle inequality, we have 
$$\beta_k\e -3\e \leq d(x_{\tau}, z_k) \leq \beta_k\e +2\e, \quad \textrm{ for all }k=1,\cdots,L.$$
For any $y\in M_{\tau}^*$ if nonempty, then the inequality above is also satisfied for $y$ in place of $x_{\tau}$ by definition. Using $\beta_k \e \geq r_0/8$ for all $k$, one similarly see that $d(y,B)>r_0/32$ for sufficiently small $\e, r_L$ depending on $r_0$.
Hence, applying Proposition \ref{coordinate-Lip} proves the claim $d(x_{\tau},y)\leq 5C\sqrt{L}\e<\CHau \e$ for any $y\in M_{\tau}^*$.
\end{proof}

\begin{figure}[h]
  \begin{center}
    \includegraphics[width=0.65\linewidth]{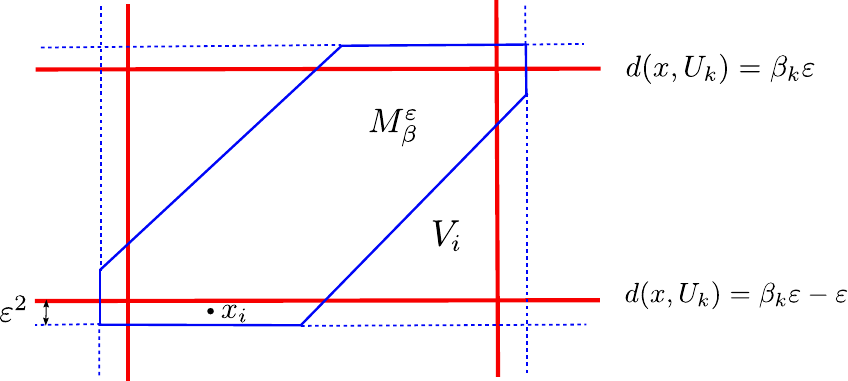}
    \caption{An illustration of sets $V_i$ and corresponding points $x_i$ (for $L=2$). The middle region enclosed by red lines is $V_i$ of the form $M_{\tau}^*$ for $\tau=(\beta_1,\beta_2,0,\cdots,0)\in \mathcal{A}$.
    The multi-index $\beta$ satisfies the criterion \eqref{beta-cri} and hence $M_{\beta \langle l \rangle}^{\e} \neq \emptyset$ for all $l$.
     The corresponding point $x_i$ to $V_i$ is chosen as any one point of $M_{\beta \langle l \rangle}^{\e}$ in the blue region enlarged by scale $\e^2$.
    It could happen that $V_i$ may be empty for some $i$, or $x_i$ may not belong to $V_i$.}
    \label{fig_slicing} 
  \end{center}
\end{figure}

\begin{definition} \label{def-Vi}
Now let us change notations. We denote by $V_i$ a set of the form $M_{\tau}^*$ (possibly empty set) for some $\tau\in \mathcal{A}$, where the index $i$ ranges over $1,\cdots, I_L$. 
The total number of such sets is uniformly bounded $I_L\leq (D/\e)^L$.
For any $V_i$ corresponding to $M_{\tau}^*$ for some $\tau\in \mathcal{A}$, let $x_i\in X$ be one point in $X$ such that Lemma \ref{lemma-proximity}(2) holds, see Figure \ref{fig_slicing}.
Note that the choice of $x_i$ for each $V_i$ may not be unique, since there could be two different multi-indices in $\mathcal{R}_{out}^*$ whose first same $L$ indices give the corresponding multi-index for $V_i$.
For our purposes, for any given $\tau\in \mathcal{A}$, we take only one of such multi-indices and take its corresponding point in $X$ in the sense of Definition \ref{def-corres}.
In particular, the base point $p\in X$ as remarked in Definition \ref{def-corres}.
\end{definition}

Proposition \ref{prop-metric} and Lemma \ref{lemma-proximity} give the following statement.

\begin{proposition} \label{partition-L}
Let $(M,g) \in \M(n,D,K,K_1,v_0)$ and $U\subset M$ be an open set containing a ball $B(p, 2r_0)$ of radius $0<r_0<{\rm inj}(M)/4$. 
Let $q$ be a potential function on $M$ satisfying \eqref{bound-q-positive}.
Then for any $\e>0$, there exists $\delta_1>0$, such that any $\delta_1$-approximation $\{\lambda_j^a,\phi_j^a |_U\}$ of the spectral data of $-\Delta_g +q$ on $U$ determine a collection of (possibly empty) sets $\{V_i\}_{i=1}^{I_L}$ and numbers $\h{d}_{ik}$ for $i,k\in \{0,\cdots, I_L\}$ satisfying the following properties.
\begin{itemize}
\item[(1)] The sets $V_i$ are disjoint for different $i$, and their union covers $M\setminus B(p,r_0)$.

\item[(2)] There exists a set of points $\{x_i\}_{i=1}^{I_L}$ with $I_L\leq (D/\e)^{L}$, satisfying $V_i\subset B(x_i, \CHau \e)$ for all $i=1,\cdots,I_L$, such that 
$$\big| \h{d}_{ik}-d(x_i,x_k) \big| \leq \Cd \e^{\frac18}, \quad \textrm{ for all } i,k=1, \cdots, I_L,$$
where $L$ is uniformly bounded depending on $n,D,K,K_1,v_0,r_0$.

\item[(3)] The distance between $x_i$ and the base point $p$ can be approximated by
$$\big| \h{d}_{0 i}-d(p,x_i) \big| \leq \Cd \e^{\frac18}, \quad \textrm{ for all } i=1, \cdots, I_L.$$
\end{itemize}
\end{proposition}

Now we approximate the first eigenfunction $\phi_1$ and its Laplacian $\Delta_g \phi_1$ on $\{x_i\}_{i=1}^{I_L}$ for $x_i\in M\setminus B(p,r_0)$.
From now on, we write 
\begin{equation}
\phi=\phi_1.
\end{equation}

\subsection{Approximating the integral of $\phi_1$}

Let $\e <r_0/64$ be fixed, and
we choose parameters
\begin{equation}\label{def-parameters}
 \rho_1=\varepsilon^{\frac{1}{16}}, \quad \rho_2=\varepsilon^{\frac{1}{64n}}.
\end{equation}
We choose sufficiently small $\e$ such that $\rho_1,\rho_2<r_0/16 < \inj(M)/64$.

\begin{lemma} \label{pointwise}
For $i=1,\cdots,I_L$ such that $d(x_i,p)\geq r_0+\rho_1$,
we define
\begin{equation} \label{def-phihat}
\h{\phi}(x_i) :=\Big(\frac{1}{\nu_n \rho_1^n} \sum_{j: \,\h{d}_{ij}<\rho_1} \int_{V_j} \phi^2 \Big)^{\frac12},
\end{equation}
where $\nu_n$ is the volume of the unit ball in $\mathbb{R}^n$, and
$\h{d}_{ij}$ is the number determined in Proposition \ref{partition-L}.
Then for sufficiently small $\e$, we have
$$\big| \h{\phi}(x_i)-\phi(x_i) \big| \leq C(n,D,K,\CL,\Cd, \ceigen,\|\phi\|_{C^{0,1}}) \, \e^{\frac{1}{16}}.
$$
\end{lemma}

\begin{proof}
First, we claim that
\begin{equation}\label{partition-gap}
B(x_i, \rho_1- \Cd \e^{\frac18}-\CL\e)\subset \bigcup_{j: \, \h{d}_{ij}<\rho_1} V_j \subset B(x_i, \rho_1+ \Cd \e^{\frac18}+\CL \e).
\end{equation}
In particular, this shows $\h{\phi}(x_i)>0$ since $V_j$ are disjoint.
Let us verify the first inclusion: the second one is simpler. For any $\xi\in B(x_i,\rho_1- \Cd \e^{\frac18}-\CL\e)$, we see $d(\xi,p)>r_0$ since $d(x_i,p)\geq r_0+\rho_1$.
Then by Proposition \ref{partition-L}(1), there exists $V_k$ such that $\xi\in V_k$. 
Proposition \ref{partition-L}(2) yields that $d(\xi,x_k)<\CL\e$.
It then follows that $d(x_i,x_k)<\rho_1-\Cd \e^{\frac18}$ and thus $\h{d}_{ik}<\rho_1$.

The relation \eqref{partition-gap} implies that
$$\Big| {\Vol} ( B(x_i,\rho_1) ) -  {\Vol}\Big(\bigcup_{j: \, \h{d}_{ij}<\rho_1} V_j  \Big) \Big| \leq  C(n,D,K,\CL,\Cd) \rho_1^{n-1} \e^{\frac18}.$$
Since $V_i$ are disjoint by Proposition \ref{partition-L}(1),
\begin{eqnarray*}
\Big| \int_{\cup_j V_j} \phi^2 - \int_{B(x_i,\rho_1)} \phi^2(x_i) \Big| &\leq& \Big| \int_{\cup_j V_j} \phi^2 - \int_{B(x_i,\rho_1)} \phi^2 \Big| + \Big| \int_{B(x_i,\rho_1)} \phi^2 - \int_{B(x_i,\rho_1)} \phi^2(x_i) \Big| \\
&\leq& C(n,D,K,\CL,\Cd) \|\phi\|_{C^0}^2 \rho_1^{n-1} \e^{\frac18}+ C(n,K)\|\phi^2\|_{Lip}  \,\rho_1^{n+1},
\end{eqnarray*}
where $\|\cdot \|_{Lip}$ denotes the Lipschitz seminorm.
Since $|{\rm vol}(B(x_i,\rho_1))-\nu_n \rho_1^n|\leq C(n,K)\rho_1^{n+2}$, we have
$$\big| \h{\phi}(x_i)^2-\phi(x_i)^2 \big| \leq C (\rho_1+\rho_1^{-1}\e^{\frac18}) \leq C(n,D,K,\CL,\Cd, \|\phi\|_{C^{0,1}}) \e^{\frac{1}{16}}.$$
As $\phi=\phi_1\geq \ceigen$ by Lemma \ref{lemma-eigen-positive}, then for sufficiently small $\e$,  we see that $\h{\phi}(x_i)\geq \ceigen/2$, and the claim follows.
\end{proof}

From now on, we assume that $\e$ is sufficient small so that 
\begin{equation}
\h{\phi}(x_i)\geq \ceigen/2>0,
\end{equation}
for all $i=1,\cdots,I_L$ satisfying $d(x_i,p)\geq r_0+\rho_1$, where $\ceigen$ is the uniform constant determined in Lemma \ref{lemma-eigen-positive}.

\begin{lemma}\label{volume-small}
For $i=1,\cdots,I_L$ such that $d(x_i,p)\geq r_0+\rho_1$,
we define
\begin{equation}
{\Vol}^a(V_i):= (\h{\phi}(x_i))^{-2} \int_{V_i} \phi^2.
\end{equation}
Then for sufficiently small $\e$, we have
$$\Big|\frac{{\Vol}^a(V_i)}{{\Vol}(V_i)} -1 \Big| \leq C(n,D,K,\CL,\Cd, \ceigen,\|\phi\|_{C^{0,1}}) \, \e^{\frac{1}{16}},\quad \textrm{ if }\, \Vol(V_i) \neq 0.$$
Note that if $\Vol(V_i) = 0$, then ${\Vol}^a(V_i)=0$ by definition.
\end{lemma}

\begin{proof}
Denote $|V_i|={\rm vol}(V_i)$. Suppose that $|V_i|\neq 0$.
Since $V_i\subset B(x_i, \CHau \e)$ due to Proposition \ref{partition-L}(2), by Lemma \ref{pointwise},
\begin{eqnarray*}
|V_i|^{-1} \int_{V_i}\phi^2 &\leq& |V_i|^{-1} \int_{V_i} \phi^2(x_i)+ \|\phi^2\|_{Lip} \, \CL \e \\
&=&  \phi^2(x_i)+ \|\phi^2\|_{Lip}\, \CL\e \leq \h{\phi}(x_i)^2+ \|\phi^2\|_{Lip}\, \CL\e+C\e^{\frac{1}{16}}.
\end{eqnarray*}
Since $\h{\phi}(x_i) \geq \ceigen/2>0$ for sufficiently small $\e$, dividing $\h{\phi}(x_i)^2$ from both sides, we have
$$\frac{{\Vol}^a(V_i)}{{\Vol}(V_i)} -1  \leq C \e^{\frac{1}{16}}.$$
The other direction of the claim is proved in the same way.
\end{proof}

\begin{lemma} \label{integral}
Let $(\CL+\Cd)\e^{\frac18}<\rho<\inj(M)/2$.
For $i=1,\cdots,I_L$ such that $d(x_i,p)\geq r_0+2\rho+\rho_1$,
one can approximate $\int_{B(x_i,\rho)} \phi$ with
$$ \sum_{j: \, \h{d}_{ij}<\rho}  (\h{\phi}(x_j))^{-1} \int_{V_j} \phi^2= \sum_{j: \, \h{d}_{ij}<\rho} {\Vol}^a(V_j) \h{\phi}(x_j),$$
and the error is uniformly bounded by $C(n,D,K,\CL,\Cd, \ceigen,\|\phi\|_{C^{0,1}}) \, \e^{\frac{1}{16}}$.
\end{lemma}

\begin{proof}
First, note that for any $j$ such that $\h{d}_{ij}<\rho$, we see that $d(x_i,x_j)<\rho+\Cd \e^{\frac18}<2\rho$ by Proposition \ref{partition-L}. This gives $d(x_j,p)>r_0+\rho_1$, so $\h{\phi}(x_j)$ and $\Vol^a(V_j)$ are defined.
Using Lemma \ref{volume-small} and the fact that $V_j$ are disjoint by Proposition \ref{partition-L}(1), we have
\begin{eqnarray*}
\Big| \sum_{j: \, \h{d}_{ij}<\rho} \int_{V_j} \h{\phi}(x_j) - \sum_{j: \, \h{d}_{ij}<\rho} {\Vol}^a(V_j) \h{\phi}(x_j) \Big| &\leq&  \sum_{j: \, \h{d}_{ij}<\rho} |{\Vol}^a(V_j)-{\Vol}(V_j)| \h{\phi}(x_j) \\
&\leq& C\e^{\frac{1}{16}} \sum_{j: \, \h{d}_{ij}<\rho} {\Vol}(V_j) \h{\phi}(x_j) \\
&\leq& C\Vol(M) \|\phi\|_{C^0} \, \e^{\frac{1}{16}}.
\end{eqnarray*}
Note that the inequalities above still hold if some of $V_j$ have zero volume, in which case $\Vol^a(V_j)$ is also zero by definition.
On the other hand, by Lemma \ref{pointwise} and $V_j\subset B(x_j, \CHau \e)$ due to Proposition \ref{partition-L}(2),
\begin{eqnarray*}
\Big| \int_{V_j} \h{\phi}(x_j)- \int_{V_j} \phi \Big| &\leq& \Big| \int_{V_j} \h{\phi}(x_j)- \int_{V_j} \phi(x_j) \Big| + \Big| \int_{V_j} \phi(x_j)- \int_{V_j} \phi \Big| \\
&\leq& C\Vol(V_j) \e^{\frac{1}{16}} + \|\phi\|_{Lip}\, \CL \Vol(V_j) \e.
\end{eqnarray*}
Using the fact that $V_j$ are disjoint, we have
$$\Big| \sum_{j: \, \h{d}_{ij}<\rho} \int_{V_j} \h{\phi}(x_j)- \sum_{j: \, \h{d}_{ij}<\rho} \int_{V_j} \phi\Big| \leq C\e^{\frac{1}{16}} \sum_{j: \, \h{d}_{ij}<\rho} \Vol(V_j) \leq C\Vol(M) \,\e^{\frac{1}{16}}.$$
Again, note that these inequalities hold if some of $V_j$ have zero volume.
Finally, using \eqref{partition-gap},
$$\Big| \int_{B(x_i,\rho)} \phi - \sum_{j: \, \h{d}_{ij}<\rho} \int_{V_j} \phi \Big| \leq C \rho^{n-1}\e^{\frac18},$$
where the dependency of constants is the same as in Lemma \ref{pointwise}.
Combining all the estimates above proves the claim.
\end{proof}

\subsection{Approximating the Laplacian of $\phi_1$}

Next, let us approximate the volume of a ball of small radius using Lemma \ref{volume-small}.

\begin{lemma} \label{volume-large}
Let $(\CL+\Cd)\e^{\frac18}<\rho<\inj(M)/2$.
For $i=1,\cdots,I_L$ such that $d(x_i,p)\geq r_0+2\rho+\rho_1$,
we define 
\begin{equation} \label{def-volume-large}
{\Vol}^a \big(B(x_i,\rho) \big):= \sum_{j: \, \h{d}_{ij}<\rho} {\Vol}^a(V_j).
\end{equation}
Then for sufficiently small $\e$, we have
$$\Big| {\Vol}^a(B(x_i,\rho))- {\Vol}(B(x_i,\rho)) \Big|  \leq C(\rho^n \e^{\frac{1}{16}} + \rho^{n-1}\e^{\frac18}),$$
where the constant $C$ depends on $n,D,K,\CL,\Cd, \ceigen,\|\phi\|_{C^{0,1}}$.
In particular, if $\rho\geq \rho_1=\e^{\frac{1}{16}}$, then ${\Vol}^a (B(x_i,\rho))$ is bounded below by $\nu_n \rho^n/4$ for sufficiently small $\e$.
\end{lemma}

\begin{proof}
Using Lemma \ref{volume-small}, \eqref{partition-gap} and the fact that $V_j$ are disjoint, we have
\begin{eqnarray*}
{\Vol}^a(B(x_i,\rho))&\leq& (1+C\e^{\frac{1}{16}})\sum_{j: \, \h{d}_{ij}<\rho_1} {\Vol}(V_j) \\
&=& (1+C\e^{\frac{1}{16}}) {\Vol}(\bigcup_{j: \, \h{d}_{ij}<\rho_1} V_j) \\
&\leq& (1+C\e^{\frac{1}{16}}) {\Vol}\big(B(x_i, \rho+ C \e^{\frac18}) \big).
\end{eqnarray*}
Hence,
$$ {\Vol}^a(B(x_i,\rho))- {\Vol}(B(x_i,\rho))  \leq C\rho^n \e^{\frac{1}{16}} + C\rho^{n-1}\e^{\frac18}.$$
The other direction of the claim is proved in the same way.
\end{proof}

Lemma \ref{volume-large} has the following variation.

\begin{coro} \label{volume-1}
Choosing $\rho=\rho_2=\e^{\frac{1}{64n}}$ in Lemma \ref{volume-large}, then
$$\Big| \frac{1}{{\Vol}^a(B(x_i,\rho_2))}- \frac{1}{{\Vol}(B(x_i,\rho_2))} \Big|  \leq C(n,D,K,\CL,\Cd, \ceigen,\|\phi\|_{C^{0,1}}) \,\e^{\frac{1}{32}}.$$
\end{coro}

\begin{proof}
This is a straightforward computation from Lemma \ref{volume-large} with the specific choice of $\rho=\rho_2$.
Indeed, the left-hand side is bounded by $C\rho_2^{-2n}(\rho_2^n \e^{\frac{1}{16}}+\rho_2^{n-1}\e^{\frac18})\leq C\rho_2^{-(n+1)} \e^{\frac{1}{16}}$. Then taking $\rho_2=\e^{\frac{1}{64n}}$ gives the estimate.
\end{proof}

Finally, let us define the graph Laplacian.

\begin{definition}[Graph Laplacian] \label{def-graph-Lap}
Let $(\CL+\Cd)\e^{\frac18}<\rho<\inj(M)/2$.
For $i=1,\cdots,I_L$ such that $d(x_i,p)\geq r_0+2\rho+\rho_1$,
we define
\begin{equation}
(\Delta_X \h{\phi})(x_i) := \frac{2(n+2)}{{\Vol}^a(B(x_i,\rho)) \rho^{2}} \Big(\sum_{j: \, \h{d}_{ij}<\rho}  (\h{\phi}(x_j))^{-1} \int_{V_j} \phi^2 \Big) -\frac{2(n+2)}{\rho^2} \h{\phi}(x_i),
\end{equation}
where $\h{\phi}(x_j)$ is defined in \eqref{def-phihat}, ${\Vol}^a(B(x_i,\rho))$ is defined in \eqref{def-volume-large}, and the numbers $\h{d}_{ij}$ are determined in Proposition \ref{partition-L}.
\end{definition}

\begin{proposition} \label{appro-Lap}
Let $(M,g) \in \M(n,D,K,K_2,v_0)$. 
Set $\rho=\rho_2=\e^{\frac{1}{64n}}$ in Definition \ref{def-graph-Lap}.
Then for $i=1,\cdots,I_L$ such that $d(x_i,p)\geq r_0+3\rho_2$, we have
$$\big| (\Delta_X \h{\phi})(x_i) - (\Delta_g \phi)(x_i) \big| \leq  \CLap \, \e^{\frac{1}{80n}},$$
where the constant $\CLap$ depends on $n,D,K,K_2, v_0,\CL,\Cd, \ceigen, \|q\|_{C^{0,1}}$.
\end{proposition}

\begin{proof}
First, we show that $(\Delta_X \h{\phi})(x_i)$ with $\rho=\rho_2$ approximates
\begin{equation} \label{eq-Lap}
\frac{2(n+2)}{{\Vol}(B(x_i,\rho_2)) \rho_2^{2}} \int_{B(x_i,\rho_2)} \phi -\frac{2(n+2)}{\rho_2^2} \phi(x_i)
= \frac{2(n+2)}{{\Vol}(B(x_i,\rho_2)) \rho_2^{2}} \int_{B(x_i,\rho_2)} \big( \phi(y) -\phi(x_i) \big) dy,
\end{equation}
with uniform error of order $\e^{\frac{1}{64}}$ for all $i$.
This directly follows from previous lemmas.
More precisely, replacing $1/{\Vol}^a(B(x_i,\rho_2))$ with $1/{\Vol}(B(x_i,\rho_2))$ introduces error of order $\e^{\frac{1}{32}}\rho_2^{-2}\leq \e^{\frac{1}{64}}$ by Corollary \ref{volume-1}. Replacing $\sum_{j: \, \h{d}_{ij}<\rho_2}  (\h{\phi}(x_j))^{-1} \int_{V_j} \phi^2$ with $\int_{B(x_i,\rho_2)} \phi$ introduces error of order $\e^{\frac{1}{16}}\rho_2^{-(n+2)}\leq \e^{\frac{1}{32}}$ by Lemma \ref{integral}. Replacing $\h{\phi}(x_i)$ with $\phi(x_i)$ introduces error of order $\e^{\frac{1}{16}}\rho_2^{-2}\leq \e^{\frac{1}{32}}$ by Lemma \ref{pointwise}.

Next, we show that \eqref{eq-Lap} approximates the Laplacian $\Delta_g \phi$ at $x_i$. 
Let $x\in M$ be arbitrary. In geodesic normal coordinate,
\begin{eqnarray*}
\frac{2(n+2)}{{\Vol}(B(x,\rho_2)) \rho_2^{2}} \int_{B(x,\rho_2)} \big( \phi(y) -\phi(x) \big) dy = \frac{2(n+2)}{{\Vol}(B(x,\rho_2)) \rho_2^{2}} \int_{\mathcal{B}(0,\rho_2)} \big( \phi(\exp_{x} v) -\phi(x) \big) J_x (v) dv,
\end{eqnarray*}
where $\mathcal{B}(0,\rho_2)$ is the ball in the tangent space $T_x M$, $\exp_x$ is the Riemannian exponential map at $x$, and $J_x(v)$ is the Jacobian of the exponential map at $v\in \mathcal{B}(0,\rho_2)$.
Let us denote
\begin{equation}
\tilde{\phi}=\phi \circ \exp_x : \mathcal{B}(0,\rho_2)\to \R. 
\end{equation}
The function $\tilde{\phi}$ has uniformly bounded $C^{2,\alpha}$-norm for any $\alpha\in (0,1)$. 
Indeed, under the additional curvature bounds \eqref{bound-C2}, there exists a uniform radius in which the components of the metric tensor in geodesic normal coordinates at any point have uniformly bounded $C^{2}$-norm, see e.g. \cite[Lemma 8]{HV}.
Using the form of Laplacian in local coordinates, $q\in C^{0,1}(M)$ and $\lambda_1 \leq \|q\|_{C^0}$, the interior H\"older estimate holds for $\tilde{\phi}$ by \cite[Thm. 8.24]{GT}.
Then the Schauder estimate gives an estimate for the $C^{2,\alpha}$-norm of $\tilde{\phi}$ for any $\alpha\in (0,1)$, see e.g. \cite[Thm 6.6]{GT}.
Now we write the function $\tilde{\phi}$ in Taylor expansion,
\begin{eqnarray*}
\int_{\mathcal{B}(0,\rho_2)} \big( \phi(\exp_{x} v) -\phi(x) \big) J_x (v) dv &=&
\int_{\mathcal{B}(0,\rho_2)} \big( \tilde{\phi}(v) -\tilde{\phi}(0) \big) J_x (v) dv \\
&=&\int_{\mathcal{B}(0,\rho_2)} \Big( \nabla \tilde{\phi} (0) \cdot v + \frac12 \nabla^2 \tilde{\phi} (0) v\cdot v + R_2(v)  \Big) J_x (v) dv,
\end{eqnarray*}
where $R_2(v)$ is the reminder which is uniformly bounded by $C(n) \|\tilde{\phi}\|_{C^{2,\alpha}} \rho_2^{2+\alpha}$.

For the first-order term, since $\int_{\mathcal{B}(0,\rho_2)} \nabla \tilde{\phi}(0) \cdot v dv=0$ due to symmetry, then replacing $J_x(v)$ by $J_x(v)-1$, we see that
$$\big|\int_{\mathcal{B}(0,\rho_2)} \nabla \tilde{\phi}(0) \cdot v J_x(v) dv \big|= \big|\int_{\mathcal{B}(0,\rho_2)} \nabla \tilde{\phi}(0) \cdot v (J_x(v)-1) dv \big| \leq C\rho_2^{n+3},$$
where we have used the Jacobian estimate $|J_x(v)-1|\leq C(n,K)|v|^2$ for geodesic normal coordinates.
We remark that this higher order of $|v|^2$ is crucial for our estimates, which is not available in harmonic coordinates. This is why we have to sacrifice additional regularity conditions on the curvature tensor.
The second-order term gives the Laplacian of $\tilde{\phi}$ at the origin of $\R^n$ (see e.g. \cite[Sec. 2.3]{BIK}):
$$\int_{\mathcal{B}(0,\rho_2)}  \frac12 \nabla^2\tilde{\phi}(0) v \cdot v  dv = \frac{\nu_n \rho_2^{n+2}}{2(n+2)} (\Delta \tilde{\phi})(0). $$
Fixing $x\in M$ and an orthonormal frame $\{e_1,\cdots,e_n\}$ at $T_x M$,
the Laplacian $(\Delta \tilde{\phi})(0)$ appears here is simply $\sum_{j=1}^n e_j e_j \tilde{\phi}$ at the origin. Since in geodesic normal coordinates the components of the metric tensor satisfy $g_{jk}(0)=\delta_{jk}$ and $\Gamma_{jk}^l (0)=0$ at the center point for all $j,k,l$, then 
\begin{eqnarray*}
(\Delta_g \phi)(x)&=& \sum_{j,k}(g^{jk} \partial_j \partial_k \phi) (x)- \sum_{j,k,l}(g^{jk} \Gamma_{jk}^l \partial_l \phi) (x) \\
&=& \sum_{j=1}^n (\partial_j \partial_j \phi)(x) = \sum_{j=1}^n (e_j e_j \tilde{\phi})(0)=(\Delta \tilde{\phi})(0).
\end{eqnarray*}
Hence, by similarly splitting $J_x(v)$ into $J_x(v)-1+1$ in the second-order term of the Taylor expansion, we have
\begin{eqnarray*}
\Big|\int_{B(x,\rho_2)} \big( \phi(y) -\phi(x) \big) dy- \frac{\nu_n \rho_2^{n+2}}{2(n+2)} (\Delta_g \phi)(x) \Big| \leq C(\|\tilde{\phi}\|_{C^{2,\alpha}})\rho_2^{n+2+\alpha},
\end{eqnarray*}
for any $0<\alpha<1$,
which gives
$$\Big|\frac{2(n+2)}{{\Vol}(B(x,\rho_2)) \rho_2^2}\int_{B(x,\rho_2)} \big( \phi(y) -\phi(x) \big) dy- \frac{\nu_n \rho_2^n}{{\Vol}(B(x,\rho_2)) } (\Delta_g \phi)(x) \Big| \leq C(\|\tilde{\phi}\|_{C^{2,\alpha}})\rho_2^{\alpha}\, .$$
Since
$|\frac{\nu_n \rho_2^n}{{\Vol}(B(x,\rho_2)) }-1| \leq C(n,K)\rho_2^2$, we have
\begin{equation} \label{Lap-appro}
\Big| (\Delta_g \phi)(x)-\frac{2(n+2)}{{\Vol}(B(x,\rho_2)) \rho_2^2}\int_{B(x,\rho_2)} \big( \phi(y) -\phi(x) \big) dy \Big| \leq C(\|\tilde{\phi}\|_{C^{2,\alpha}})\rho_2^{\alpha} + C(\|\Delta_g \phi\|_{C^0}) \rho_2^2\,,
\end{equation}
for any $0<\alpha<1$ and any $x\in M$. 
The $C^{2,\alpha}$-norm of $\tilde{\phi}$ and hence the $C^0$-norm of $\Delta_g \phi$ are uniformly bounded depending on $n,D,K,K_2,v_0,\|q\|_{C^{0,1}}$.
Then taking $x=x_i$ and $\alpha=4/5$ complete the proof.
\end{proof}

\begin{remark} \label{remark-Lap}
The discrete operator defined in Definition \ref{def-graph-Lap} is known as a type of weighted graph Laplacians on proximity graphs studied in e.g. \cite{BIK,BIK2,BIKL,Fuji,Lu}. The points $x_i$ at which the operator is defined form a $\CHau\e$-net of $M\setminus B(p,3r_0/2)$; they are the vertices of the graph. We say two points $x_i,x_j$ are connected by an edge if and only if $\h{d}_{ij}<\rho_2$; they form the edges of the graph. 
Our constructions at the beginning of Section \ref{sec-potential} provide a (disjoint) partition $V_i$ such that $V_i\subset B(x_i,\CHau\e)$.
Then the operator in Definition \ref{def-graph-Lap} is roughly of the form
$$ \frac{C(n)}{\Vol(B(x_i,\rho)) \rho^2} \sum_{j: \, \h{d}_{ij} <\rho} \Vol(V_j) \big( \h{\phi}(x_j)-\h{\phi}(x_i) \big),$$
where $C(n)$ is a normalization constant.
The volumes of $V_j$ and $B(x_i, \rho)$ can be approximated by $\Vol^a(V_j)$ and $\Vol^a(B(x_i,\rho))$ up to uniform error by Lemma \ref{volume-small} and \ref{volume-large}.
We note that one cannot simply use the approximation of $\Vol(B(x_i,\rho))$ by $\nu_n \rho^n$ here, because the error term would be in the second order, the same order as the Laplacian in the Taylor expansion, causing the discrete operator not approximating the integral form \eqref{eq-Lap}.
\end{remark}

\subsection{Approximation by spectral data}

The next step is to argue that approximations of the spectral data determine all relevant quantities in this section with sufficient precision. 
Let $\e>0$ be fixed.
Recall that $V_i$ are of the form $M_{\tau}^*$ for some $\tau\in \mathcal{A}$, where $\tau$ is of the form $(\tau_1,\cdots,\tau_L,0,\cdots,0)$, see the beginning of Section \ref{sec-potential}.
In the same way as Lemma \ref{norm-beta}, one has the following statement.
For any $\sigma\in (0,1)$, there exists sufficiently small $\delta=\delta(\sigma,\e)$, such that 
by knowing a $\delta$-approximation $\{\lambda_j^a,\phi_j^a |_U\}$ of the spectral data of $-\Delta_g +q$ on $U$,
we can compute a number $\appnorm(V_i)$ satisfying
\begin{equation} \label{error-Vi}
\Big| \appnorm(V_i) - \int_{V_i} \phi_1^2 \,\Big| \leq 2^{L}\sigma,
\end{equation}
where $L>1$ is bounded above by a uniform constant.
The key point is that $\sigma$ can be arbitrary small independent of $\e$.
Thus, one can make all relevant quantities in arbitrary small error independent of $\e$.

More precisely, one can take 
\begin{equation} \label{choice-sigma}
\sigma=2^{-L} \e^{4L},
\end{equation}
and find the corresponding $\delta$.
Let $\{x_i\}_{i=1}^{I_L}$ be the set of points stated in Proposition \ref{partition-L}.
We pick a subset of these points such that 
\begin{equation} \label{choice-xi}
\h{d}_{0i} \geq r_0+3\rho_2+\Cd \e^{\frac18},
\end{equation}
where $\rho_2$ is set in \eqref{def-parameters}.
Then Proposition \ref{partition-L}(3) yields that $d(x_i,p)\geq r_0+3\rho_2$.
We claim that the $\CHau\e$-neighborhood of this subset of points covers $M\setminus B(p,3r_0/2)$.
Indeed, for any $y\in M\setminus B(p,3r_0/2)$, there exists $V_k$ such that $y\in V_k$ and $d(y,x_k)<\CHau\e$ for the corresponding point $x_k$ by Proposition \ref{partition-L}. 
Observe that all such $x_k$ must have been picked up by the criterion \ref{choice-xi}, since it satisfies $d(p,x_k)>3r_0/2 - \CHau\e$. Because for points $x_i$ satisfying $\h{d}_{0i} < r_0+3\rho_2+\Cd \e^{\frac18}$, Proposition \ref{partition-L}(3) would yield $d(p,x_i)<r_0+3\rho_2+2\Cd \e^{\frac18}<5r_0/4$, which is a contradiction.

Let us fix the choice of $\sigma$ in \eqref{choice-sigma} and find all $x_i$ satisfying \eqref{choice-xi}.
Examining Lemma \ref{pointwise}, as the numbers $\h{d}_{ij}$ have been determined, one can determine $\h{\phi}(x_i)^2$ up to error of order $(D/\e)^L \e^{4L}\leq C\e^{3L}$. This is because the total number of sets $V_j$ is bounded by $(D/\e)^L$.
Then one determines $\Vol^a(V_i)$ in Lemma \ref{volume-small} up to error of order $\e^{3L}$.
Note that the same statement holds if $\Vol(V_i)=0$ for some $i$.
Hence, one determines $\int_{B(x_i,\rho_2)} \phi $ in Lemma \ref{integral} and $\Vol^a(B(x_i,\rho_2))$ in Lemma \ref{volume-large} up to error of order $(D/\e)^L \e^{3L}\leq C\e^{2L}$.
In particular, this determines $1/\Vol^a(B(x_i,\rho_2))$ up to error of order $\e^{2L}/\rho_2^{2n}< \e^{2L-1}$.
At last, combining all these estimates, we conclude that we can determine the graph Laplacian $(\Delta_X \h{\phi})(x_i)$ in Definition \ref{def-graph-Lap} up to error of order $\e^{2L-2}$.
This is higher order error relative to the estimate obtained in Proposition \ref{appro-Lap}.

\smallskip
Finally, using Proposition \ref{appro-Lap}, Lemma \ref{pointwise}, and the equation $-\Delta_g \phi_1+q \phi_1 =\lambda_1 \phi_1$, we can determine
\begin{equation} \label{def-hqi}
\h{q}_i:= \lambda_1^a + \frac{(\Delta_X \h{\phi})(x_i)}{\h{\phi}(x_i)},
\end{equation}
which approximates $q(x_i)$ up to error of order $\delta+\e^{\frac{1}{80n}}$, in view of Definition \ref{def-data-close}.
Requiring $\delta<\e$ without loss of generality, we have proved the following result.

\begin{proposition} \label{prop-potential}
Let $(M,g) \in \M(n,D,K,K_2,v_0)$ and $U\subset M$ be an open set containing a ball $B(p,2r_0)$ of radius $0<r_0<{\rm inj}(M)/4$.
Let $q$ be a potential function on $M$ satisfying \eqref{bound-q-positive}.
Then for any $\e>0$, there exists $\delta_2>0$, such that 
any $\delta_2$-approximation $\{\lambda_j^a,\phi_j^a |_U\}$ of the spectral data of $-\Delta_g + q$ on $U$ determine the numbers $\h{q}_i$, $i=1,\cdots,I_1$, such that the following holds.

\begin{itemize}
\item[(1)] There exists a set of points $\{x_i\}_{i=1}^{I_1}$ such that
$$\big| \, \h{q}_i - q(x_i) \big| \leq \Cpot \, \e^{\frac{1}{80n}}.$$

\item[(2)] The points satisfy $x_i\in M\setminus B(p,r_0)$ for all $i$, and
the $\CHau\e$-neighborhood of the set $\{x_i\}_{i=1}^{I_1}$ covers $M\setminus B(p,3r_0/2)$.
\end{itemize}
The constant $\Cpot$ depends on $n,D,K,K_2,v_0,r_0,\Cq$, and
the choice of $\delta_2$ depends on the same set of parameters and additionally on $\e$.
\end{proposition}

\section{Proofs of main results} \label{sec-proofs}

Now let us prove the main results.

\begin{proof}[Proof of Theorem \ref{main1}]

The first part of the result is proved in Proposition \ref{prop-metric}.
In Proposition \ref{prop-potential}, we constructed an approximation of the potential $q$ on a $\CHau\e$-net of $M\setminus B(p,3r_0/2)$.
Note that the choices of $\delta_1,\delta_2$ in the reconstruction of metric and potential are different: in fact $\delta_2 \ll \delta_1$.
To prove Theorem \ref{main1}, it suffices to construct an approximation of the potential on a $\CHau\e$-net of $B(p,2r_0)$.
Using the given Riemannian metric on $U$, this can be done directly from \eqref{Lap-appro} and Definition \ref{def-data-close} as follows.

Without loss of generality, assume that the Riemannian metric is given on $B(p,5r_0)\subset U$ with $r_0<\inj(M)/10$.
Denote 
\begin{equation}
\mathcal{L}_{\rho}(x):=\frac{2(n+2)}{{\Vol}(B(x,\rho)) \rho^2}\int_{B(x,\rho)} \big( \phi_1(y) -\phi_1(x) \big) dy,
\end{equation}
where $\phi_1>0$ is the first eigenfunction.
We have proved in \eqref{Lap-appro} that, for any $x\in M$,
\begin{equation}
\big| (\Delta_g \phi_1)(x)-\mathcal{L}_{\rho}(x) \big| \leq \CLap (\alpha) \rho^{\alpha},
\end{equation}
for sufficiently small radius $\rho$ and for any $0<\alpha<1$. 
Let us take $x\in B(p,2r_0)$ and $\alpha=1/2$.
Since we are given the metric on $B(p,5r_0)$, we know all quantities in $\mathcal{L}_{\rho}(x)$ except the values of $\phi_1$.
We can approximate $\mathcal{L}_{\rho}(x)$ using any $\delta$-approximation $\{\lambda_j^a, \phi_j^a|_U \}$ of the spectral data on $U$. Namely, we define
\begin{equation}
\mathcal{L}^a_{\rho}(x):=\frac{2(n+2)}{{\Vol}(B(x,\rho)) \rho^2}\int_{B(x,\rho)} \big( \phi^a_1(y) -\phi^a_1(x) \big) dy,
\end{equation}
and hence by Definition \ref{def-data-close},
$$\big| \mathcal{L}^a_{\rho}(x) - \mathcal{L}_{\rho}(x)\big| \leq C(n)\delta \rho^{-2}, \quad \textrm{ for }x\in B(p,2r_0).$$
Therefore, similar to \eqref{def-hqi}, we can evaluate
\begin{equation}
\h{q}(x) := \lambda_1^a + \frac{\mathcal{L}^a_{\rho}(x)}{\phi_1^a (x)}, \quad \textrm{ for }x\in B(p,2r_0).
\end{equation}
Note that we can require $\delta<c_1/2$ so that $\phi_1^a >c_1/2>0$ on $U$, where $\ceigen$ is the uniform constant determined in Lemma \ref{lemma-eigen-positive}.
Choosing $\delta<\e$ and $\rho=\e^{\frac14}$, we see that
\begin{equation}
| \,\h{q}(x) - q(x) | \leq C(n, \ceigen, \CLap, \|\phi_1\|_{C^0}, \|q\|_{C^0}) \,\e^{\frac18}, \quad \textrm{ for }x\in B(p,2r_0).
\end{equation}
Thus, taking any $\CHau\e$-net of $B(p,2r_0)$ completes the proof of Theorem \ref{main1}.
\end{proof}

\begin{proof}[Proof of Theorem \ref{main-coro}]

Let $\{\lambda_j,\phi_j|_{U_2}\}_{j=1}^{\infty}$ be the spectral data of $-\Delta_{g_2}+q_2$ on $U_2\subset M_2$.
By Theorem \ref{main1}(1),
for any $\e>0$, the finite spectral data $\{\lambda_j, \phi_j |_{U_2}\}_{j=1}^J$ for $J>\delta^{-1}$ determine a Hausdorff approximation $\mathcal{R}^*$ of the interior distance functions on $B_2 (p,r_0/2)\subset M_2$,
which determines a Riemannian manifold $\widehat{M}$ that is diffeomorphic to $M_2$, and the Lipschitz distance between $M_2$ and $\widehat{M}$ is bounded by $C_1\e^{1/12}$.
Let $F:U_1\to U_2$ be the Riemannian isometry. 
By Definition \ref{def-data-close}, the pull-back of spectral data $\{\lambda_j, (F^* \phi_j)|_{U_1}\}_{j=1}^J$ is a $\delta$-approximation of the spectral data of $-\Delta_{g_1}+q_1$ on $U_1\subset M_1$.
The pull-back of the spectral data via isometry $F$ produces the same Hausdorff approximation $\mathcal{R}^*$ of the interior distance functions on $B_1(F^{-1}(p),r_0/2)\subset M_1$, which determines the same Riemannian manifold $\widehat{M}$ that is now diffeomorphic to $M_1$ with Lipschitz distance uniformly bounded by $C_1\e^{1/12}$.
Therefore, $M_1$ is diffeomorphic to $M_2$, with their Lipschitz distance bounded by $2C_1\e^{1/12}$.

Furthermore, by Theorem \ref{main1}(2), the finite spectral data $\{\lambda_j, \phi_j |_{U_2}\}_{j=1}^J$ on $U_2$ and its pull-back $\{\lambda_j, (F^* \phi_j)|_{U_1}\}_{j=1}^J$ on $U_1$ determine the same set of numbers $\{\h{q}_i\}_{i=1}^{I_0}$ such that the following holds. There exist $\CHau\e$-nets $\{y_i\}_{i=1}^{I_0}\subset M_2$ and $\{x_i\}_{i=1}^{I_0} \subset M_1$ satisfying that 
$$|\h{q}_i - q_1(x_i)| \leq C_2 \e^{\frac{1}{80n}},\quad |\h{q}_i - q_2(y_i)| \leq C_2 \e^{\frac{1}{80n}},\quad \textrm{ for all }i=1,\cdots,I_0.$$
Hence, for respective points in the nets, we have
$$| q_1(x_i)- q_2(y_i)| \leq  2C_2 \e^{\frac{1}{80n}}, \quad \textrm{ for all }i=1,\cdots,I_0.$$
Since $\{x_i\}_{i=1}^{I_0}$ is a $\CHau\e$-net in $M_1$, we take any (disjoint) partition $\{\Omega_i\}_{i=1}^{I_0}$ of the manifold such that $\Omega_i\subset B(x_i,\CHau\e)$ for all $i$.
For example, one can take the Voronoi decomposition of $M_1$ with respect to $\{x_i\}_{i=1}^{I_0}$.
Then we define a (non-continuous) map $\Psi:M_1 \to M_2$ by 
\begin{equation} \label{def-Psi}
\Psi(z) :=y_i, \quad \textrm{ if }z\in \Omega_i.
\end{equation}
Hence, for any point $z\in \Omega_i$, since $q_1$ is Lipschitz, we have
$$\big| q_1(z)-(q_2\circ \Psi)(z) \big| \leq |q_1(z)-q_1(x_i)|+|q_1(x_i)-q_2(y_i)| \leq \Cq \CHau \e+ 2C_2 \e^{\frac{1}{80n}}.$$
The estimate holds uniformly for $z$ and $i$.
Tracing the dependence of all the parameters, the modulus of continuity of $\e$ in $\delta$ is double logarithmic.
The two logarithms both come from the quantitative unique continuation, Theorem \ref{uc-Y}.
The considerations in Section \ref{sec-coordinate}, and the choice of $\sigma$ in \eqref{sigma1} and \eqref{choice-sigma} avoid a third logarithm.
One can refer to a similar computation in \cite[Appendix A]{BILL}.

The map $\Psi$ defined in \eqref{def-Psi} is actually a $(2\Cd \e^{\frac18}+2\CHau \e)$-isometry. By the proof of Theorem \ref{main1}, the $\CHau\e$-net $\{x_i\}_{i=1}^{I_0}$ in $M_1$ above is the union of the points in $X$ defined in Definition \ref{def-Vi} and an arbitrary $\CHau\e$-net in $B(p,2r_0)$.
We can simply choose the net in $B(p,2r_0)$ to be the one $\e$-net $\Gamma_{\e}$ taken in Definition \ref{def-rbeta}(2).
Then we see that $\{x_i\}_{i=1}^{I_0}$ is contained in the set $X$ of corresponding points chosen in Definition \ref{def-corres}, and any $\delta$-approximation of the spectral data determine numbers $\h{d}_{ik}$, for $i,k\in \{1,2,\cdots,I_0\}$, which approximate $d_{M_1}(x_i,x_k)$ up to error $\Cd \e^{\frac18}$ by Proposition \ref{prop-metric}.
On $M_2$, the same constructions for the net $\{y_i\}_{i=1}^{I_0}$ determine the same numbers $\h{d}_{ik}$ that approximate $d_{M_2}(y_i,y_k)$ up to error $\Cd \e^{\frac18}$.
Hence,
$$\big| d_{M_1}(x_i,x_k)-d_{M_2}(y_i,y_k) \big| \leq 2\Cd \e^{\frac18},\quad \textrm{ for all }i,k=1,\cdots,I_0.$$
Then for any $x,x'\in M_1$, since $\Omega_i\subset B(x_i,\CHau\e)$, we have
$$\big| d_{M_2} (\Psi(x),\Psi(x')) -d_{M_1}(x,x') \big| < | d_{M_2}(y_i,y_k)-d_{M_1}(x_i,x_k)|+ 2\CHau\e \leq 2\Cd \e^{\frac18}+ 2\CHau \e,$$
for some $i,k$. This shows $\Psi$ is a $(2\CHau \e+2\Cd \e^{\frac18})$-isometry.
\end{proof}

\bigskip

\end{document}